\documentclass[11pt, a4paper]{amsart}
\usepackage[dvips]{epsfig}
\usepackage{amsmath}
\usepackage{amssymb}
\usepackage{amsfonts}
\usepackage{amsthm}
\usepackage{amsbsy}
\usepackage{amsgen}
\usepackage{amscd}
\usepackage{amsopn}
\usepackage{amstext}
\usepackage{amsxtra}
\usepackage{epsf}

\newtheorem{theorem}{Theorem}[section]
\newtheorem{proposition}[theorem]{Proposition}
\newtheorem{lemma}[theorem]{Lemma}
\newtheorem{corollary}[theorem]{Corollary}

\theoremstyle{definition}

\newtheorem{remark}[theorem]{Remark}
\newtheorem{notation}[theorem]{Notation}

\newcommand {\Z} {\mathbb{Z}}

\newcommand {\R} {\mathbb{R}}

\newcommand {\E} {\mathbb{E}}
\newcommand {\Var} {{\rm Var}}

\newcommand {\N} {\mathbb {N}}

\title[Distribution of zeros]{The distribution of the zeroes of random trigonometric polynomials}
\author{Andrew Granville and Igor Wigman}
\address{D\'epartement de math\'ematiques et de statistiques,
Universit\'e de Montr\'eal C.P. 6128, succ. centre-ville Montr\'eal,
Qu\'ebec H3C 3J7, Canada.
 } \email{andrew@dms.umontreal.ca}
\address{Centre de recherches math\'ematiques (CRM),
Universit\'e de Montr\'eal C.P. 6128, succ. centre-ville Montr\'eal,
Qu\'ebec H3C 3J7, Canada \newline
currently at \newline
Institutionen f\"{o}r Matematik, Kungliga Tekniska h\"{o}gskolan (KTH),
Lindstedtsv\"{a}gen 25, 10044 Stockholm, Sweden} \email{wigman@kth.se}
\thanks{I.W. is supported by a CRM ISM fellowship, Montr\'eal and
the Knut and Alice Wallenberg Foundation, grant KAW.2005.0098, Stockholm}

\date{}

\begin{document}
\maketitle

\begin{abstract}
We study the asymptotic distribution of the number $Z_{N}$ of zeros
of random trigonometric polynomials of degree $N$ as
$N\rightarrow\infty$. It is known that as $N$ grows to infinity, the
expected number of the zeros is asymptotic to
$\frac{2}{\sqrt{3}}\cdot N$. The asymptotic form of the variance was
predicted by Bogomolny, Bohigas and Leboeuf to be $cN$ for some
$c>0$. We prove that $\frac{Z_{N}-\E Z_{N}}{\sqrt{cN}}$ converges to
the standard Gaussian. In addition, we find that the analogous
result is applicable for the number of zeros in short intervals.
\end{abstract}

\section{Introduction}

The distribution of zeros of random functions for various ensembles
is one of the most studied problems. Of the most significant and
important among those is the ensemble of random trigonometric
polynomials, as the distribution of it zeros occurs in a wide range
of problems in science and engineering, such as nuclear physics (in
particular, random matrix theory), statistical mechanics, quantum
mechanics, theory of noise etc.

\subsection{Background}

Understanding the distribution of zeros of random functions was
first pursued by Littlewood and Offord ~\cite{LO1}, ~\cite{LO2} and
~\cite{LO3}. They considered, in particular, the distribution of the
number of real roots of polynomials
\begin{equation}
\label{eq:rand pol} P_N(x) = a_0+a_1x+\ldots+a_N x^N,
\end{equation}
of degree $N$ with random coefficients $a_n$, as
$N\rightarrow\infty$. For the coefficients $a_n$ taking each of the
values $1,-1$ with equal probability $1/2$, they showed that the
number $Z_{P_{N}}$ of zeros of $P_{N}(x)$ satisfies
\begin{equation}
\label{eq:rand pol exp} Z_{P_{N}} \sim \frac{2}{\pi}\ln{N}
\end{equation}
for $\big(1-o_{N\rightarrow\infty}(1)\big)2^{N}$ of the vectors
$\vec{a}\in \{\pm 1\}^{N}$. Later, Erdos and Offord ~\cite{EO}
refined their estimate.

% ADD a paragraph with the results of Kac, Ibragimov and Ibragimov and Maslova
Kac ~\cite{K} proved that the expected number of zeros $Z_{P_{N}}$
of the random polynomials \eqref{eq:rand pol} of degree $N$, this
time $a_n$ being Gaussian i.i.d. with mean $0$ and variance $1$, is
asymptotic to the same expression \eqref{eq:rand pol exp}. His
result was generalized by Ibragimov and Maslova ~\cite{IM1} and
~\cite{IM2}, who treated any distributions of the coefficients
$a_n$, provided that they belong to the domain of attraction of the
normal law: if each $\E a_n = 0$ then the expectation is again
asymptotic to \eqref{eq:rand pol exp}, though, if $\E a_n \ne 0$ one
expects half as many zeros as in the previous case, that is
\begin{equation*}
\E Z_{P_{N}} \sim \frac{1}{\pi}\ln{N}.
\end{equation*}

Maslova ~\cite{M1} also established the only heretofore known
asymptotics for the variance of the number of {\em real} zeros $Z$,
\begin{equation*}
\Var Z_{P_{N}} \sim \frac{4}{\pi}(1-\frac{2}{\pi})\cdot \ln{N}
\end{equation*}
for the ensemble \eqref{eq:rand pol} of random functions. In her
next paper ~\cite{M2}, she went further to establish an even more
striking result, proving the normal limiting distribution for
$Z_{P_{N}}$, as $N\rightarrow\infty$.

The case of random {\em trigonometric} polynomials was considered by
Dunnage ~\cite{DN}. Let $T_{N}:[0,2\pi]\rightarrow\R$ be defined by
\begin{equation}
\label{eq:trig pol T def} T_{N}(t) = \sum\limits_{n=1}^{N} a_n
\cos{nt},
\end{equation}
where $a_{n}$ are standard Gaussian i.i.d, and $Z_{T_{N}}$ be the
number of zeros of $T_{N}$ on $[0,2\pi]$. Dunnage proved that as
$N\rightarrow\infty$, $\E Z_{T_{N}}$ is asymptotic to $$\E
Z_{T_{N}}\sim \frac{2}{\sqrt{3}} N,$$ and, moreover, that the
distribution is concentrated around the expectation in the same
sense as Littlewood and Offord mentioned earlier.

The variance of the zeros for \eqref{eq:trig pol T def} was shown by
Farahmand to be $O(N^{3/2})$ in ~\cite{F}.
This estimate implies that the distribution of $Z_{T_{N}}$
concentrates around the mean.

Qualls ~\cite{Q} considered a slightly different class of
trigonometric polynomials,
\begin{equation*}
X_{N}(t) = \frac{1}{\sqrt{N}}\sum\limits_{n=1}^{N}
\bigg(a_n\sin{nt}+b_n\cos{nt}\bigg).
\end{equation*}
Let $Z_{X_{N}}$ be the number of the zeros of $X_{N}$ on $[0,2\pi]$.
Applying the theory of stationary processes on $X_{N}$, one finds
that
\begin{equation*}
\E Z_{X_{N}} = 2\sqrt{\frac{(N+1)(2N+1)}{6}}\sim \frac{2}{\sqrt{3}}
N,
\end{equation*}
similar to \eqref{eq:trig pol T def}. Qualls proved that
\begin{equation*}
\bigg| Z_{X_{N}} - \E Z_{X_{N}} \bigg| \le C \cdot N^{3/4}
\end{equation*}
for some $C>0$
with probability $1-o_{N\rightarrow\infty}(1)$.

Bogomolny, Bohigas and Leboeuf ~\cite{BBL} argued that the
{\em variance} of $Z_{X_{N}}$ satisfies
\begin{equation*}
\Var(Z_{X_{N}}) \sim cN,
\end{equation*}
as $N\rightarrow\infty$, where $c$ is a positive constant
approximated by
\begin{equation}
\label{eq:c approx 0.558...} c\approx 0.55826
\end{equation}
(this is equivalent to the formula (3.34) in ~\cite{BBL} and the numeric value of
$\Delta\approx 0.44733$
immediately afterwards; one should bear in mind that they normalize the random variable
to have unit expectancy).

In this paper we study the distribution of the random variable
$Z_{X_{N}}$ in more detail. We will find the asymptotics of the
variance $\Var(Z_{X_{N}})$ as well as prove the central limit
theorem for the distribution of $Z_{X_{N}}$ (see section
\ref{sec:stat of res}). We guess, but have not proved, that the same
result may be true for Dunnage's ensemble \eqref{eq:trig pol T def}.

The zeros of random {\em complex analytic} functions were examined
in a series of papers by Sodin-Tsirelson (see e.g. ~\cite{ST}), and
Shiffman-Zelditch (see e.g. ~\cite{SZ}).

\subsection{Statement of results}
\label{sec:stat of res}

Let $X_{N}:[0,2\pi]\rightarrow\R$ be Qualls' ensemble
\begin{equation}
\label{eq:XN def} X_{N}(t) = \frac{1}{\sqrt{N}}\sum\limits_{n=1}^{N}
a_n\sin{nt}+b_n\cos{nt},
\end{equation}
where $a_n$ and $b_n$ are standard Gaussian i.i.d.

As usual, given a random variable $Y$, we define $\E(Y)$ to be the
expectation of $Y$. For example, for any {\em fixed} $t\in [0,2\pi]$
and $N$, one has
$$\E(X_{N}(t)^2) = 1. $$

Above we noted Qualls' result that
\begin{equation}
\label{eq:exp num zer} \E(Z_{X_{N}}) = 2\sqrt{\lambda_2},
\end{equation}
where
\begin{equation*}
\lambda_2 := \frac{1}{N}\sum\limits_{n=1}^{N} n^2 =
\frac{(2N+1)(N+1)}{6}.
\end{equation*}

We prove the central limit theorem for $Z_{X_{N}}$:

\begin{theorem}
\label{thm:main res}

There exists a constant $c>0$ such that the distribution of
\begin{equation*}
\frac{Z_{X_{N}} - \E Z_{X_{N}}}{\sqrt{c N}}
\end{equation*}
converges weakly to the standard Gaussian $N(0,1)$. The variance
is asymptotic to
\begin{equation}
\label{eq:var Z sim} {\Var} (Z_{X_{N}}) \sim cN,
\end{equation}
as $N\rightarrow\infty$, as predicted by Bogomolny-Bohigas-Leboeuf.

\end{theorem}

We can compute the value of the constant $c$ in Theorem
\ref{thm:main res} as
\begin{equation*}
c = \frac{4}{3\pi}c_0+ \frac{2}{\sqrt{3}},
\end{equation*}
with
\begin{equation}
\label{eq:c0 def} c_0=\int\limits_{0}^{\infty} \bigg[
\frac{(1-g(x)^2)-3g'(x)^2}{(1-g(x)^2)^{3/2}}(\sqrt{1-{R^*}^2}+{R^*}\arcsin{R^*})
 - 1 \bigg],
\end{equation}
where we denote
\begin{equation}
\label{eq:g def} g(x) := \frac{\sin{x}}{x}
\end{equation}
and
\begin{equation}
\label{eq:R* def} R^* (x) := \frac{g''(x)(1-g(x)^2)+g(x)g'(x)^2
}{\frac{1}{3}(1-g(x)^2) - g'(x)^2}.
\end{equation}

More generally, for $0\le a < b \le 2\pi$ one defines $Z_{X_{N}}
(a,b)$ to be the number of zeros of $X_{N}$ on the subinterval
$[a,b]\subseteq [0,2\pi]$. It is easy to generalize the computation
of the expectation \eqref{eq:exp num zer} for this case as $$\E
Z_{X_{N}}(a,b) = \frac{\sqrt{\lambda_{2}}}{\pi} \cdot (b-a).$$

A priori, it seems that the behaviour of the number of zeros of
$X_{N}$ in {\em short} intervals $[a_{N},b_{N}]$, shrinking as
$N\rightarrow\infty$, should be more erratic than on the full
interval. Surprisingly, just as in the previous case, we are able to
find a precise asymptotics for the variance $\Var
Z_{N}(a_{N},b_{N})$, and prove a central limit theorem, provided
that $[a_{N},b_{N}]$ does not shrink too rapidly. We have the
following Theorem:

\begin{theorem}
\label{thm:main res short} Let $0 \le a_{N} < b_{N} \le 2\pi$ be any
sequences of numbers with $N\cdot (b_{N}-a_{N}) \rightarrow \infty$.
Then as $N\rightarrow\infty$,
\begin{equation*}
{\Var} (Z_{X_{N}}(a_{N},b_{N})) \sim c \cdot
\frac{(b_{N}-a_{N})}{2\pi}N,
\end{equation*}
where $c$ is the same constant as in Theorem \ref{thm:main res}.
Moreover,
\begin{equation*}
\frac{Z_{X_{N}}(a_{N},b_{N}) - \E Z_{X_{N}}(a_{N},b_{N})}{\sqrt{c
\frac{(b_{N}-a_{N})}{2\pi}N}}
\end{equation*}
converges weakly to the standard Gaussian $N(0,1)$.
\end{theorem}

The proof of Theorem \ref{thm:main res short} is identical to the
proof of Theorem \ref{thm:main res}, and in this paper we will give
only the proof of Theorem \ref{thm:main res}.

\subsection{Plan of the paper}

To prove the central limit theorem we will first need to establish
the asymptotic form \eqref{eq:var Z sim} for the variance. This is
done throughout sections \ref{sec:sec mom int} and \ref{sec:asmp
var}: in section \ref{sec:sec mom int} we develop an integral
formula for the second moment of the number of zeros, and in section
\ref{sec:asmp var} we exploit it to study the asymptotics of the
variance.

Sections \ref{sec:proof of CLT} and \ref{sec:pr E(Z-Zmol)^2=o(N)} are dedicated to the proof of the main
statement of Theorem \ref{thm:main res}, that is the central limit
theorem. While the proof is contained in section \ref{sec:proof of CLT}, a
certain result, required by the proof, is proven throughout section
\ref{sec:pr E(Z-Zmol)^2=o(N)}.

\subsection{On the proof of Theorem \ref{thm:main res}}

As an initial step for the central limit theorem, we will have to
find the asymptotics \eqref{eq:var Z sim} for the variance. This is
done throughout sections \ref{sec:sec mom int} and \ref{sec:asmp
var}.

While computing the asymptotic of the variance of $Z_{X_{N}}$, we
determined that the {\em covariance function} $r_{N}$ of $X_{N}$ has
a scaling limit $r_{\infty}(x)=\frac{\sin{x}}{x}$, which proved
useful for the purpose of computing the asympotics. Rather than
scaling $r_{N}$, one might consider scaling $X_{N}$.

We realize, that the above {\em should} mean, that the distribution
of $Z_{X_{N}}$ is intimately related to the distribution of the
number $\tilde{Z}_{N}$ of the zeros on (roughly) $[0,N]$ of a
certain Gaussian stationary process $Y(x)$, defined on the real line
$\R$, with covariance function $r=r_{\infty}$ (see section
\ref{sec:phil rem}). Intuitively, this should follow, for example,
from the approach of ~\cite{GS}, see e.g. Theorem 9.2.2, page 450.
Unfortunately, this approach seems to be difficult to make rigorous,
due to the different scales of the processes involved.

The latter problem of the distribution of the number of the zeros
(and various other functionals) on growing intervals is a classical
problem in the theory of stochastic processes. Malevich ~\cite{ML}
and subsequently Cuzick ~\cite{CZ} prove the central limit theorem
for $\tilde{Z}_{N}$, provided that $r$ lies in some (rather wide)
class of functions, which include $r_{\infty}$. Their result was
generalized in a series of papers by Slud (see e.g. ~\cite{SL}), and
the two-dimensional case was treated by Kratz and Leon ~\cite{KL}.

We modify the proof of Malevich-Cuzick to suit our case. There are
several marked differences between our case and theirs. In their
work, one has to deal with growing sums of identically distributed
(but by no means independent) random variables (which will be
referred to as a {\em linear system}); to prove the central limit
theorem one applies a result due to Diananda ~\cite{DN}. In our
case, we deal with {\em triangular systems} (to be defined),
applying a theorem of Berk ~\cite{BR}. For more details about the
proof, see section \ref{sec:abt proof main res CLT}.

\subsection{Acknowledgements}

The second author wishes to thank Ze\'{e}v Rudnick for suggesting the
problem as well as his help and support while conducting the
research. In addition, he wishes to thank Mikhail Sodin, Ildar A.
Ibragimov, P\"{a}r Kurlberg and Iosif Polterovich for many fruitful
discussions. We are grateful to Phil Sosoe for conducting some
empirical experiments available in the longer version of this paper.
We wish to thank Jonathan Keating for pointing out ~\cite{BBL}.
The authors wish to thank the anonymous referees for many useful
comments and suggestions how to improve the paper.

\section{A formula for the second moment}
\label{sec:sec mom int}

\begin{proposition}
\label{prop:sec mom int} We have
\begin{equation}
\label{eq:sec mom int} \E(Z_{X_{N}}^2)-\E(Z_{X_{N}}) =
\frac{2}{\pi}\int\limits_{0}^{2\pi} \frac{\lambda_2(1-r^2) -
(r')^2}{(1-r^2)^{3/2}} \bigg(\sqrt{1-\rho^2}+\rho\arcsin{\rho}\bigg)
dt,
\end{equation}
where
\begin{equation}
\label{eq:lambda def} \lambda_2 = \lambda_{2,N}=
\frac{1}{N}\sum\limits_{n=1}^{N}n^2 = \frac{(N+1)(2N+1)}{6},
\end{equation}
\begin{equation}
\label{eq:r def} r(t)=r_{X_{N}}(t) =
\frac{1}{N}\sum\limits_{n=1}^{N} \cos{nt} = \frac{1}{2N} \bigg[
\frac{\sin{ (N+1/2)t} - \sin{t/2}}{\sin{(t/2)}}\bigg]
\end{equation}
and
\begin{equation}
\label{eq:rho def} \rho = \rho_{N}(t) = \frac{r''(1-r^2)+(r')^2
r}{\lambda_{2}(1-r^2)-(r')^2}.
\end{equation}

\end{proposition}
A similar but less explicit formula was obtained by Steinberg et al.
~\cite{SSWZ}. The rest of this section is dedicated to the proof of
this result.

The ensemble $X_{N}$ is a centered stationary Gaussian process
(meaning that the finite dimensional distributions are Gaussian with
zero mean).
An explicit computation with the double angle formula shows that
its covariance function is
\begin{equation*}
r_{X_{N}}(t_{1},t_{2}) = r_{X_{N}}(t_{2}-t_{1}),
\end{equation*}
with the function on the right side as defined in \eqref{eq:r def}.

Let $I$ be an interval and $X:I\rightarrow \R$ be a mean zero stationary process with
covariance function $r$. We assume that $r(0)=1$ (i.e. $X$ has unit variance) and
furthermore that
the sample functions of $X$ are a.s. sufficiently smooth (e.g. twice differentiable)
so that its sets of zeros is discrete. We have
\begin{equation}
\label{eq:|r|<=1}
|r(t)|\le 1
\end{equation}
for every $t\in I$ by the Cauchy-Schwartz inequality.
We denote $Z$ to be the number of zeros of $X$ on $I$; $Z$ is a random variable.

In general, we have
the following celebrated Kac-Rice formula (see e.g. ~\cite{CL}) for
the expectation of $Z$:
\begin{equation*}
\E Z = \frac{|I|}{\pi} \sqrt{\lambda_{2}},
\end{equation*}
where $|I|$ is the length of $I$ (finite or infinite), and
$\lambda_{2}=-r''(0)$. As mentioned earlier, it was exploited
by Qualls to compute the expected number \eqref{eq:exp num zer}
of zeros of trigonometric polynomials.

In this section we find a formula for the second moment $\E Z_{X}^2$
of the number of zeros of any Gaussian stationary process $X$ on $I$,
assuming that its covariance function $r$ is smooth.
To determine $\E(Z_{X}^{2})$, we naturally encounter the distribution
of the random vector
\begin{equation}
\label{eq:vec val der t1t2} V = V_{t_1,t_2} := (X(t_1), X(t_2),
X'(t_1), X'(t_2)).
\end{equation}
for some {\em fixed} $t_{1},t_{2}\in I$. As an artifact of the stationarity
of $X$, the distribution of $V$
depends only on $t:=t_2-t_1$. The covariance matrix of $V$ is
\begin{equation}
\label{eq:cov mat} \Sigma =  \Sigma(t) := \left( \begin{matrix} 1 & r(t) &0 & r'(t)  \\
r(t) & 1 &-r'(t) &0 \\ 0 &-r'(t) &\lambda_2 &-r''(t) \\ r'(t) &0
&-r''(t) &\lambda_2
\end{matrix} \right)
=: \left( \begin{matrix} A & B \\ B^t &C\end{matrix} \right).
\end{equation}
The random vector $V$ has a multivariate normal distribution with mean
zero and covariance matrix $\Sigma$. For $t\in (0, 2\pi)$, $\Sigma(t)$ is
nonsingular (for $X_{N}$ see ~\cite{Q} and Remark \ref{rem:CL mom form just}).

\begin{lemma}
\label{lem:CL form zeros} Let $X$ be a Gaussian stationary process,
which almost surely has a continuous sample derivative such that
the distribution of $V$ is nondegenerate
for $t_{1}\ne t_{2}$. Then
\begin{equation}
\label{eq:CL form zeros}
\begin{split}
&\E(Z_{X}^2)-\E(Z_{X}) \\&= \iint\limits_{[0,\, 2\pi]\times [0,\,
2\pi]} dt_1 dt_2 \iint\limits_{\R^2} |y_1| |y_2|
\phi_{t_1,t_2}(0,0,y_1,y_2)dt_1 dt_2,
\end{split}
\end{equation}
where $\phi_{t_1,t_2}(u_1, u_2,v_1,v_2)$ is the probability density function
of $V_{t_{1},t_{2}}$.
\end{lemma}

\begin{remark}
\label{rem:CL mom form just} The original formulation of
Lemma \ref{lem:CL form zeros} from ~\cite{CL} assumes that
the Fourier transform of $r$ has a continuous
component, a stronger condition than stated.
However, their proof works just as well in the less restrictive
case as formulated above. Qualls proved that the trigonometric
polynomials \eqref{eq:XN def} satisfy this assumption and thus we
may apply Lemma \ref{lem:CL form zeros} to \eqref{eq:XN def}.
Qualls' argument can be generalized to higher moments: we use it to
bound the third moment in Proposition \ref{prop:third mom ZM bnd}.
\end{remark}

\begin{remark}
\label{rem:VarZ=intE|X'(0)X'(t)||0} Let $\psi_{t_{1},t_{2}}$ be the
probability density function of the random vector
$(X'(t_{1}),X'(t_{2}))$ conditional on $
X(t_{1})=X(t_{2})=0$. Then we have
\begin{equation}
\label{eq:phi=psi/det} \phi_{t_{1},t_{2}}(0,0,y_{1},y_{2}) =
\frac{\psi_{t_{1},t_{2}}(y_{1},y_{2})}{2\pi
\sqrt{1-r(t_{2}-t_{1})^2}}
\end{equation}
(see also \eqref{eq:det Sigma Jacobi}). Therefore we may rewrite
\eqref{eq:CL form zeros} as
\begin{equation*}
\E(Z_{X}^2)-\E(Z_{X}) = \iint\limits_{I\times I} \frac{\E
\big[|X'(t_{1})X'(t_{2}) | \big| X(t_{1})=X(t_{2})=0\big]}{
\sqrt{1-r(t_{2}-t_{1})^2}} \frac{dt_{1} dt_{2}}{2\pi}.
\end{equation*}
We use this representation in the proof of Proposition
\ref{prop:E(Z-Zmol)^2=o(N)}, as well as its analogue for the third
moment in the proof of Proposition \ref{prop:third mom ZM bnd}.
\end{remark}

We use Lemma \ref{lem:CL form zeros} to derive the following.
\begin{corollary}
\label{cor:EZ^2-EZ int expl} Under the assumptions of Lemma
\ref{lem:CL form zeros}, one has
\begin{equation}
\label{eq:EZ^2-EZ int expl} \E(Z_{X_{N}}^2)-\E(Z_{X_{N}}) =
\iint\limits_{I\times I} \frac{\lambda_2(1-r^2) -
(r')^2}{(1-r^2)^{3/2}} (\sqrt{1-\rho^2}+\rho\arcsin{\rho})
\frac{dt_1 dt_2}{\pi^2},
\end{equation}
where $r=r_{X}(t_{2}-t_{1})$, and $\rho=\rho_{X}(t_{2}-t_{1})$ with
\begin{equation*}
\rho_{X}(t) = \frac{r''(t)(1-r(t)^2)+r'(t)^2
r(t)}{\lambda_{2}(1-r(t)^2)-r'(t)^2}.
\end{equation*}

\end{corollary}

\begin{proof}

Direct matrix multiplication confirms that
\begin{equation*}
\Sigma^{-1} = \left(\begin{matrix} (A-BC^{-1}B^{t})^{-1}& -A^{-1}B(C-B^{t}A^{-1}B)^{-1} \\
-C^{-1}B^{t}(A-BC^{-1}B^{t})^{-1}& (C-B^{t}A^{-1}B)^{-1}
\end{matrix}\right),
\end{equation*}
so if $\Omega$ is the $2\times 2$ ``reduced covariance matrix", that
is $\Omega^{-1}$ is the bottom right corner of $\Sigma^{-1}$, then
\begin{equation}
\label{eq:Omega blocks} \Omega = C-B^t A^{-1} B.
\end{equation}
The matrix $\Omega$ is the covariance matrix of the random vector
$(X'(t_{1}),X'(t_{2}))$ conditioned upon $X(t_{1})=X(t_{2})=0$.

Computing \eqref{eq:Omega blocks} explicitly, we have
\begin{equation*}
\Omega = \mu \Omega_1,
\end{equation*}
where $$\Omega_1 = \left(\begin{matrix} 1 & -\rho \\ -\rho &1
\end{matrix}\right)$$
with $\rho$ given by \eqref{eq:rho def} and
\begin{equation}
\label{eq:mu def} \mu := \frac{\lambda_2(1-r^2) - (r')^2}{1-r^2} >
0.
\end{equation}
Since $\Omega$ is a covariance matrix, we have
\begin{equation}
\label{eq:|rho|<=1}
|\rho| \le 1
\end{equation}
by the Cauchy-Schwartz inequality.

The easy to check identity
\begin{equation*}
\Sigma = \left(\begin{matrix}A &0 \\ B^{t} &I \end{matrix}
\right)\cdot \left( \begin{matrix}I &A^{-1}B \\ 0 &\Omega
\end{matrix}\right),
\end{equation*}
yields
\begin{equation}
\label{eq:det Sigma Jacobi} \det{\Sigma} =
\det{A}\det{\Omega}=(1-r^2)\mu^2 (1-\rho^2).
\end{equation}

Using \eqref{eq:CL form zeros} and the explicit form of the Gaussian density $\phi_{t_{1},t_{2}}$, we
obtain
\begin{equation}
\label{eq:EZX^2-EZX int}
\begin{split}
&\E(Z_X^2)-\E(Z_X) = \iint\limits_{I^2} dt_1 dt_2
\iint\limits_{\R^2} |y_1| |y_2|
\frac{\exp(-\frac{1}{2}y\Omega^{-1}y^t)}{\sqrt{\det{\Sigma}}}
\frac{dy_1 dy_2}{(2\pi)^2},
\\&= \iint\limits_{I^2} \frac{dt_1 dt_2}{\mu\sqrt{(1-r^2)
(1-\rho^2)}} \iint\limits_{\R^2} |y_1| |y_2| \exp\left(-
\frac{1}{2}\mu^{-1} y \Omega_1^{-1}y^t\right) \frac{dy_1 dy_2}{(2\pi)^2}
\\&= \iint\limits_{I^2} \frac{\mu}{\sqrt{(1-r^2) (1-\rho^2)}} dt_1
dt_2 \iint\limits_{\R^2} |z_1| |z_2| \exp\left(- \frac{1}{2}z
\Omega_1^{-1} z^t\right) \frac{dz_1 dz_2}{(2\pi)^2},
\end{split}
\end{equation}
where $y=(y_1,\,y_2)$, making the change of coordinates
$z=\frac{y}{\sqrt{\mu}} $. The inner integral is
\begin{equation}
\begin{split}
\label{eq:int |z1||z2| Gauss}
&\iint\limits_{\R^2}  |z_1 | |z_2| \exp\bigg(- \frac{1}{2(1-\rho^2)}
(z_1^2 +2\rho z_1 z_2 +z_2^2) \bigg) dz \\&= 4(1-\rho^2)
\bigg(1+\frac{\rho}{\sqrt{1-\rho^2}}\arcsin{\rho}\bigg),
\end{split}
\end{equation}
computed by Bleher and Di ~\cite{BD}, appendix A. Substituting this
in \eqref{eq:EZX^2-EZX int} gives our result.

\end{proof}

We are finally in the position to prove Proposition \ref{prop:sec mom int}.

\begin{proof}[Concluding the proof of Proposition \ref{prop:sec mom int}]

We use Corollary \ref{cor:EZ^2-EZ int expl} on the trigonometric
polynomials $X_{N}$. The integrand in \eqref{eq:EZ^2-EZ int expl}
depends only on $t:=t_2-t_1$ (because of the stationarity of
$X_{N}$), which allows us to convert the double integral into a
simple one. Proposition \ref{prop:sec mom int} then follows from the
periodicity of the integrand.

\end{proof}

\section{Asymptotics for the variance}
\label{sec:asmp var} Formulas \eqref{eq:sec mom int} and \eqref{eq:exp num zer} imply the
following formula for the variance
\begin{equation}
\label{eq:var J exp} \Var(Z_X) = J+\E(Z_X),
\end{equation}
where
\begin{equation}
\label{eq:int J} J:= \frac{2}{\pi}\int\limits_{0}^{2\pi}
\bigg[\frac{\lambda_2(1-r^2) - (r')^2}{(1-r^2)^{3/2}}
\bigg(\sqrt{1-\rho^2}+\rho\arcsin{\rho}\bigg) - \lambda_2 \bigg] dt
\end{equation}

\subsection{Evaluating the integral $J$ in \eqref{eq:int J}}

\begin{proposition}
\label{prop:J asymp}
\begin{equation}
\label{eq:asympt of J} J=\frac{4 c_0}{3\pi}N \bigg(1+O
\bigg(\frac{\log{N}}{N}\bigg)^{1/13}\bigg),
\end{equation}
where $c_0$ is defined by \eqref{eq:c0 def}.
\end{proposition}

Our key observation is that $r_{X_{N}}$ has a scaling limit, or,
more precisely, we have
$$f_{N}(x):=r_{X_{N}} \bigg(\frac{x}{m}\bigg) =
\frac{\sin{x}}{x}+\text{small error}$$ with $m=N+\frac{1}{2}$ and
$x\in [0,m\pi]$ (treating $[m\pi, 2 m\pi]$ by symmetry), at least outside a small interval around the
origin. It is therefore natural to change the integration variable
in \eqref{eq:int J} from $mt$ to $x$, which will recover the
asymptotics for $J$ being linear with $m$, and thus also with $N$
(see \eqref{eq:J chng var}). In fact,
rather than introducing a new parameter $m$, it is also possible to
use $x=Nt$; it results in a nastier computation.

We will argue that it is possible, up to an admissible error, to
replace the $f=f_{N}$ in the resulting integrand by $g(x):=
\frac{\sin{x}}{x}$, the latter being $N$-independent. The decay at
infinity of the new integrand (i.e. with $f$ replaced by $g$) imply
that the integral will converge to a constant intimately related to
\eqref{eq:c0 def}.

We divide the new domain of integration $[0,\pi m]$ (which is an
artifact of changing the variable of integration $x=mt$; we also use
the symmetry around $t=\pi$) into two ranges. Lemma \ref{lem:int est
t small} will bound the contribution of a small neighbourhood
$[0,\delta]$ of the origin. The main contribution to the integral
\eqref{eq:J chng var} results from $[\delta,\pi m]$, where Lemma
\ref{lem:int est t large} will allow us to replace $f_{N}$ in the
integrand with $N$-independent $g(x)=\frac{\sin{x}}{x}$.

\begin{notation}
In this manuscript we will use the notations $A\ll B$ and $A=O(B)$
interchangeably.
\end{notation}

\begin{lemma}
\label{lem:int est t small}

Let
\begin{equation}
\label{eq:M def} M(x) :=
\frac{\lambda_{2}'(1-f(x)^2)-f'(x)^2}{(1-f(x)^2)^{3/2}}(\sqrt{1-R(x)^2}+R(x)\arcsin{R(x)}),
\end{equation}
where
\begin{equation}
\label{eq:f def} f(x) =f_{N}(x)= r_{X_{N}}\bigg(\frac{x}{m}\bigg)=
\bigg(\frac{\sin{x}}{\sin{\frac{x}{2m}}} - 1\bigg)\frac{1}{2m-1},
\end{equation}
\begin{equation}
\label{eq:R def} R(x) = \rho_{N} \bigg(\frac{x}{m}\bigg) =
\frac{f''(x)(1-f(x)^2)+f(x)f'(x)^2}{\lambda_{2}'(1-f(x)^2)-f'(x)^2}
\end{equation}
and
\begin{equation}
\label{eq:lam' def} \lambda_{2}' = \frac{1+\frac{1}{2m}}{3}.
\end{equation}
There exists a universal $\delta_0 > 0$, such that for any
\begin{equation}
\label{eq:assump delta not too large} 0 < \delta < \delta_0,
\end{equation}
we have the following estimate
$$\int\limits_{0}^{\delta} M(x)dx = O(\delta^2),$$ where the constant involved
in the ``O"-notation is universal.
\end{lemma}

\begin{proof}

First we note that
\begin{equation}
\label{eq:R(x) <= 1}
|R (x)| \le 1
\end{equation}
by the definition \eqref{eq:R def} of $R(x)$ and \eqref{eq:|rho|<=1},
so that the definition of $M(x)$ makes sense.

We have to estimate $f(x)$ and its derivative around the origin.
Expanding $f$ and $f'$ into Taylor polynomial around $x=0$, we have
\begin{equation*}
\begin{split}
f(x) &= 1+a_m x^2 +b_m x^4+O(x^6),
\end{split}
\end{equation*}
and
\begin{equation*}
\begin{split}
f'(x) = 2a_m x+4b_m x^3+O(x^5).
\end{split}
\end{equation*}
with $$a_m := -\frac{2m+1}{12m}=O(1), $$ \vspace{1mm} $$b_m :=
\frac{(2m+1)(12m^2-7)}{2880 m^3}=O(1),$$ and the constants in the
`O'-notation being universal.

Thus,
\begin{equation*}
\begin{split}
1-f(x)^2 = -2a_m^2 x^2-(2b_m+a_m^2)x^4+O(x^6) \gg x^2,
\end{split}
\end{equation*}
and
\begin{equation*}
\begin{split}
\lambda_{2}'(1-f(x)^2)-f'(x)^2 &= \frac{1+\frac{1}{2m}}{3} (-2a_m
x^2-
(2b_m +a_m^2)x^4 + O(x^6)) \\&- 4x^2(a_m^2+4a_m b_m x^2 + O(x^4)) \\
= &\frac{64m^4+24m^3-108m^2-94m-21}{8640m^4} x^4 + O(x^6) \ll x^4.
\end{split}
\end{equation*}

Now
\begin{equation*}
1 \le \sqrt{1-y^2} + y\arcsin{y} \le \pi/2
\end{equation*}
for every $y\in [-1,1]$, so combining the last three displayed
equations, we obtain
\begin{equation*}
M(x) \ll \frac{x^4}{x^3} =x,
\end{equation*}
and the Lemma follows.

\end{proof}

\begin{lemma}
\label{lem:int est t large}

Let
\begin{equation}
\label{eq:assump delta not too small} \delta> (m/2)^{-1/9}
\end{equation}
and $$\delta < x \le \pi m .$$ Denote
\begin{equation*}
M_1(x) := M(x) - \lambda_{2}',
\end{equation*}
where $M(x)$ is given by \eqref{eq:M def}. Then, for $m$
sufficiently large,
\begin{equation*}
\begin{split}
M_1(x) &= \frac{1}{3} \cdot \bigg[
\frac{(1-g(x)^2)-3g'(x)^2}{(1-g(x)^2)^{3/2}}(\sqrt{1-{R^*(x)}^2}+{R^*(x)}\arcsin{R^*(x)})
 - 1 \bigg] \\&+ O\bigg(\frac{1}{\delta^{12} m x} +
\frac{1}{\delta^8m^2}\bigg),
\end{split}
\end{equation*}
where, as usual, $g(x)$ and $R^*$ are given by \eqref{eq:g def} and
\eqref{eq:R* def} respectively.
\end{lemma}

\begin{proof}
We will approximate $f$ and its first couple of derivatives by $g$
and its first couple of derivatives.

For $\delta<x<\pi m$, we have
\begin{equation}
\label{eq:f est g}
\begin{split}
f(x) &= \bigg(\frac{\sin{x}}{\frac{x}{2m}(1+O(\frac{x^2}{m^2}))} - 1
\bigg) \frac{1}{2m-1} \\&= \bigg[\frac{2m\cdot\sin(x)}{x} \bigg(
1+O\bigg(\frac{x^2}{m^2}\bigg) \bigg)- 1 \bigg] \frac{1}{2m-1}   =
g(x)+O\bigg(\frac{1}{m}\bigg),
\end{split}
\end{equation}
where to justify the second equality, we use the explicit
coefficient of the second summand in the Taylor formula for the
sine. By a straightforward computation of Taylor's formula, one easily 
obtains the sharp estimate
\begin{equation*}
f'(x)=g'(x)+O\left[ g'(x)\left(\frac{1}{m}+\frac{x^2}{m^2}   \right)  \right]
+O\left[ g(x)\frac{x}{m^2}  \right]
\end{equation*}
for $0<x<\pi m$. Noting that $g(x)=O\left( \frac{1}{|x|+1}  \right)$, 
$g'(x)=O\left( \frac{x}{x^2+1}  \right)$, we have
\begin{equation}
\label{eq:f' est g'}
\begin{split}
f'(x) &
= g'(x) + O\bigg(\frac{x}{m^2}\bigg)+O\bigg(\frac{1}{mx}\bigg)
\end{split}
\end{equation}
for $0 < x < \pi m$. For $0 < x < 1$ we have the better estimate $$f'(x) = g'(x) + O\left(\frac{x}{m} \right).$$
Similar statements hold for $f''$; in particular, for $\delta<x<\pi m$ we have
\begin{equation}
\label{eq:f'' est g''}
\begin{split}
f''(x) =  g''(x) + O\bigg(\frac{1}{\delta^2 mx}+\frac{x}{m^2}\bigg).
\end{split}
\end{equation}

Next, we apply \eqref{eq:f est g} to obtain
\begin{equation*}
\begin{split}
1-f(x)^2 =
(1-g(x)^2)\cdot\bigg(1+O\bigg(\frac{1}{\delta^2mx}\bigg)\bigg),
\end{split}
\end{equation*}
where we used for $x>\delta$, $$1-g(x)^2 \gg 1-g(\delta)^2 =
\frac{\delta^2-\sin{\delta}^2}{\delta^2} \gg \delta^2 .$$ Thus
\begin{equation}
\begin{split}
\label{eq:kr1 est g denum} \big(1-f(x)^2\big)^{-3/2} = (1-g(x)^2)^{-3/2}\cdot \bigg(1+O\bigg(\frac{1}{\delta^2mx}\bigg)\bigg),
\end{split}
\end{equation}
where we used the assumption \eqref{eq:assump delta not too small}
to bound $\frac{1}{\delta^2 mx}$ away from $1$.

By the above, we have
\begin{equation}
\begin{split}
\label{eq:kr1 est g num} \lambda_{2}'(1-f(x)^2)-f'(x)^2
&= \lambda_{2}'(1-g(x)^2) - g'(x)^2 + O\bigg(\frac{1}{\delta^2mx}\bigg)
\\&= (\lambda_{2}'(1-g(x)^2) - g'(x)^2)\bigg(1+O\bigg(\frac{1}{\delta^6mx}\bigg)\bigg),
\end{split}
\end{equation}
since for $x>\delta$,
\begin{equation}
\label{eq:kern denum low bnd} \lambda_{2}'(1-g(x)^2) - g'(x)^2 \ge
\frac{1}{3}(1-g(x)^2) - g'(x)^2 \ge \frac{1}{3}(1-g(\delta)^2) -
g'(\delta)^2 \gg \delta^4.
\end{equation}

Next, we have
\begin{equation*}
\begin{split}
f''(x)(1-f(x)^2)+f(x)f'(x)^2=
g''(x)(1-g(x)^2)+g(x)g'(x)^2 + O\bigg(\frac{1}{\delta^2 m x} +
\frac{x}{m^2}\bigg),
\end{split}
\end{equation*}
using \eqref{eq:assump delta not too small} again. Therefore,
\begin{equation*}
\begin{split}
R(x) =
\frac{g''(x)(1-g(x)^2)+g(x)g'(x)^2 }{\lambda_{2}'(1-g(x)^2) -
g'(x)^2} +O\bigg(\frac{1}{\delta^6 m x} +
\frac{x}{\delta^4m^2}\bigg),
\end{split}
\end{equation*}
exploiting \eqref{eq:assump delta not too small} once more as well
as \eqref{eq:kern denum low bnd}.

Note that
\begin{equation*}
\begin{split}
&\frac{g''(x)(1-g(x)^2)+g(x)g'(x)^2 }{\lambda_{2}'(1-g(x)^2) -
g'(x)^2} = \frac{g''(x)(1-g(x)^2)+g(x)g'(x)^2
}{\frac{1}{3}(1-g(x)^2) - g'(x)^2+O(\frac{1}{m})}
\\&= \frac{g''(x)(1-g(x)^2)+g(x)g'(x)^2 }{\frac{1}{3}(1-g(x)^2) -
g'(x)^2} + O\bigg(\frac{1}{\delta^8 mx}\bigg),
\end{split}
\end{equation*}
where we use twice \eqref{eq:kern denum low bnd} as well as
\eqref{eq:assump delta not too small} again.

All in all we obtain
\begin{equation}
\label{eq:R=R*+O()} R(x) = R^{*} (x) + O\bigg(\frac{1}{\delta^8 m x}
+ \frac{x}{\delta^4m^2}\bigg),
\end{equation}
where $R^{*}$ is defined by \eqref{eq:R* def}. It is important to
notice that \eqref{eq:R(x) <= 1} implies
$$|R^* (x)| \le 1,$$ since for any \emph{fixed} $x$,
$$R(x) \rightarrow R^* (x),$$ as $m\rightarrow\infty$ by \eqref{eq:R=R*+O()}
and \eqref{eq:assump delta not too small}.

Notice that
\begin{equation}
\label{eq:R* bnd above} R^*(x) = O\bigg(\frac{1}{\delta^4 x}\bigg),
\end{equation}
again by \eqref{eq:kern denum low bnd}. Therefore for $x\gg
1/\delta^4$, $R^{*}(x)$ is bounded away from $1$ so that for a
small $\epsilon > 0$ one has
\begin{equation*}
\begin{split}
&\sqrt{1-(R^{*}+\epsilon)^2} = \sqrt{1-R^{*}(x)}\cdot
\sqrt{1-\frac{2\epsilon R^{*}(x)+\epsilon^2}{1-R^{*}(x)^2}} \\&=
\sqrt{1-R^{*}(x)} \cdot \sqrt{1-O(\epsilon R^{*}(x)+\epsilon^2)} =
\sqrt{1-R^{*}(x)} + O(\epsilon R^{*}(x)+\epsilon^2).
\end{split}
\end{equation*}
It yields
\begin{equation*}
\sqrt{1-R(x)^2} =
 \sqrt{1-R^*(x)^2} + O\bigg(\frac{1}{\delta^{12} mx^2}+ \frac{1}{\delta^8
m^2}\bigg),
\end{equation*}
and similarly $$R(x)\arcsin{R(x)} = R^*(x) \arcsin{R^*(x)} +
O\bigg(\frac{1}{\delta^{12} mx^2}+ \frac{1}{\delta^8 m^2}\bigg).$$
Therefore for $x\gg 1/\delta^4$, we have
\begin{equation}
\label{eq:approx krn2 R*}
\begin{split}
& \sqrt{1-R(x)^2}+R(x)\arcsin{R(x)} \\&=
\sqrt{1-{R^*(x)}^2}+{R^*(x)}\arcsin{R^*(x)}+O\bigg(\frac{1}{\delta^{12}
mx^2}+ \frac{1}{\delta^8 m^2}\bigg).
\end{split}
\end{equation}
On the other hand, for $\delta < x \ll \frac{1}{\delta^4}$,
\begin{equation*}
\begin{split}
&\sqrt{1-R(x)^2}+R(x)\arcsin{R(x)} \\&=
\sqrt{1-{R^*(x)}^2}+{R^*(x)}\arcsin{R^*(x)}+
O\bigg(\frac{1}{\delta^8 m x} + \frac{x}{\delta^4m^2}\bigg) \\ &=
\sqrt{1-{R^*(x)}^2}+{R^*(x)}\arcsin{R^*(x)}+
O\bigg(\frac{1}{\delta^{12} m x} + \frac{1}{\delta^8m^2}\bigg),
\end{split}
\end{equation*}
since the derivative of the function $$x\mapsto
\sqrt{1-x^2}+x\arcsin{x},$$ namely $\arcsin{x}$, is bounded
everywhere in $[-1,\, 1]$. Therefore \eqref{eq:approx krn2 R*} is
valid for $x>\delta$.

Also we have
\begin{equation}
\label{eq:upper bnd fluct kr2} \sqrt{1-{R^*(x)}^2
}+R^*(x)\arcsin{R^* (x)} = 1+O({R^{*}(x)}^2) =
1+O\bigg(\frac{1}{\delta^8 x^2}\bigg).
\end{equation}
by \eqref{eq:R* bnd above}.

Collecting \eqref{eq:kr1 est g num} and \eqref{eq:kr1 est g denum},
we have
\begin{equation*}
\begin{split}
&\frac{\lambda_{2}'(1-f(x)^2)-f'(x)^2}{(1-f(x)^2)^{3/2}} =
\frac{\lambda_{2}'(1-g(x)^2)-g'(x)^2}{(1-g(x)^2)^{3/2}}\bigg(1+O\bigg(\frac{1}{\delta^6
mx}\bigg)\bigg)  \\
&= \frac{1/3(1-g(x)^2)-g'(x)^2}{(1-g(x)^2)^{3/2}} + \frac{1}{6m} +
O\bigg(\frac{1}{\delta^6 mx}\bigg).
\end{split}
\end{equation*}

Together with \eqref{eq:approx krn2 R*} it gives
\begin{equation*}
\begin{split}
&\frac{\lambda_{2}'(1-f(x)^2)-f'(x)^2}{(1-f(x)^2)^{3/2}}\bigg(\sqrt{1-R(x)^2}+R(x)\arcsin{R(x)}\bigg)
- \lambda_{2}'
\\ &= \frac{1/3(1-g(x)^2)-g'(x)^2}{(1-g(x)^2)^{3/2}}\bigg(\sqrt{1-{R^*(x)}^2}+{R^*(x)}\arcsin{R^*(x)}\bigg)
\\&+\frac{1}{6m}\bigg(1+O\bigg(\frac{1}{\delta^8 x^2}\bigg)\bigg) - \frac{1}{3} -
\frac{1}{6m}+ O\bigg(\frac{1}{\delta^{12} m x} +
\frac{1}{\delta^8m^2}\bigg)
\\&= \frac{1}{3} \cdot \bigg[
\frac{(1-g(x)^2)-3g'(x)^2}{(1-g(x)^2)^{3/2}}\bigg(\sqrt{1-{R^*(x)}^2}+{R^*(x)}\arcsin{R^*(x)}\bigg)
 - 1 \bigg] \\&+ O\bigg(\frac{1}{\delta^{12} m x} +
\frac{1}{\delta^8m^2}\bigg) ,
\end{split}
\end{equation*}
by \eqref{eq:upper bnd fluct kr2}.

\end{proof}

\begin{proof}[Proof of Proposition \ref{prop:J asymp}]

Noting that the integrand in \eqref{eq:int J} is symmetric around
$t=\pi$, denoting $m:=N+\frac{1}{2}$ and changing the variable of
integration $mt$ to $x$ in \eqref{eq:int J}, we find that $J$ is
\begin{equation}
\label{eq:J chng var} J = \frac{4m}{\pi} \int\limits_{0}^{\pi m}
\bigg[ \frac{\lambda_{2}'(1-f(x)^2)-f'(x)^2}{(1-f(x)^2)^{3/2}}
\bigg(\sqrt{1-R(x)^2} + R(x)\arcsin{R(x)}\bigg) - \lambda_{2}'
\bigg] dx,
\end{equation}
where $f$, $R$ and $\lambda'_{2}$ are defined in \eqref{eq:f def},
\eqref{eq:R def} and \eqref{eq:lam' def}.

We divide the interval into two ranges: $I_1 := [0, \, \delta]$ and
$I_2=[\delta,\, \pi m]$, for some parameter $\delta = \delta(m)>0$.
On $I_1$ we employ Lemma \ref{lem:int est t small} to bound (from
above) the total contribution of the integrand, whereas we invoke
Lemma \ref{lem:int est t large} to asymptotically estimate the
integral on $I_2$. The constant $\delta$ has to satisfy the
constraint of Lemma \ref{lem:int est t large}, namely
\eqref{eq:assump delta not too small}. The constraint of Lemma
\ref{lem:int est t small}, \eqref{eq:assump delta not too large}, is
satisfied for $m$ sufficiently large, provided that $\delta$
vanishes with $m$. To bound the contribution of $\lambda_{2}'$ to
the integral on $I_1$, we use the trivial estimate $\lambda_{2}' =
O(1)$.

Hence we obtain
\begin{equation}
\begin{split}
\label{eq:J=int_0^pi m g, R*}
J &= \frac{4m}{3\pi} \int\limits_{0}^{\pi m} \bigg[
\frac{(1-g(x)^2)-3g'(x)^2}{(1-g(x)^2)^{3/2}}\bigg(\sqrt{1-{R^*(x)}^2}+{R^*(x)}\arcsin{R^*(x)}\bigg)
 - 1 \bigg] \\&+ O(\frac{\log{m}}{\delta^{12} } +
\frac{1}{\delta^8}) + O(\delta m) .
\end{split}
\end{equation}
Note that for a large $x$ we have
\begin{equation*}
(1-g(x)^2)-3g'(x)^2 = 1+O\left( \frac{1}{x^2}  \right),
\end{equation*}
\begin{equation*}
\frac{1}{(1-g(x)^2)^{3/2}} = 1+O\left( \frac{1}{x^2}  \right),
\end{equation*}
and the definition \eqref{eq:R* def} implies
\begin{equation*}
R^{*}(x) = O\left( \frac{1}{x^2}  \right)
\end{equation*}
so that
\begin{equation*}
\sqrt{1-{R^*(x)}^2}+{R^*(x)}\arcsin{R^*(x)} = 1 + O(R^{*}(x)) = 1+O\left( \frac{1}{x^2}  \right).
\end{equation*}
Plugging the estimates above into the integrand of \eqref{eq:J=int_0^pi m g, R*}
shows that the integrand is $O(\frac{1}{x^2})$. Hence
\begin{equation*}
\begin{split}
J &= \frac{4N}{3\pi}
\int\limits_{0}^{\infty}
\bigg[\frac{(1-g(x)^2)-3g'(x)^2}{(1-g(x)^2)^{3/2}}\bigg(\sqrt{1-{R^*(x)}^2}+{R^*(x)}\arcsin{R^*(x)}\bigg)
 - 1 \bigg] \\&+ O(\frac{\log{N}}{\delta^{12} } +
\frac{1}{\delta^8}) + O(\delta N).
\end{split}
\end{equation*}
We finally obtain the statement
\eqref{eq:asympt of J} of the present proposition upon choosing $$\delta:=
\bigg( \frac{\log{N}}{N} \bigg) ^{1/13}.$$

\end{proof}

\subsection{Concluding the proof of the variance part \eqref{eq:var Z sim} of Theorem \ref{thm:main
res}}

\begin{proof}
Proposition \ref{prop:sec mom int} together with Proposition
\ref{prop:J asymp} and \eqref{eq:var J exp} imply
\begin{equation*}
\Var(Z_X) = J+\E(Z_X) \sim \frac{4 c_0}{3\pi} N+\frac{2}{\sqrt{3}}N=
cN.
\end{equation*}

It then remains to show that $c > 0$. As mentioned earlier,
Bogomolny-Bohigas-Lebeouf ~\cite{BBL} estimated $c\approx 0.55826$
and it is possible to use numerical methods to rigorously obtain $c>0$
(see the longer version of the present paper).

There exists a more systematic approach though. One can construct a
Gaussian process $Y_{\infty}(x)$ on $\R$ with the covariance function
$r_{\infty}(x) = \frac{\sin{x}}{x}$ (see section \ref{sec:phil
rem}). In this case, we denote again $$\lambda_{2,\infty} =
-r''_{\infty}(0) = \frac{1}{3}.$$ For $T>0$, let $Z_{\infty}(T)$ be
the number of the zeros of $Y_{\infty}$ on $[0,T]$.

By the general theory of stochastic processes developed in
~\cite{CL}, one has $$\E Z_{\infty}(T) = \frac{T}{\pi}
\lambda_{2,\infty}.$$ Using the same method we used to compute the
variance of $X_{N}$, it is not difficult to see that
$$\Var Z_{\infty}(T) \sim c T,$$ where $c$ is the same constant as in
Theorem \ref{thm:main res}, provided that $c > 0$. It was proved by
Slud ~\cite{SL}, that it is indeed the case for a wide class of
covariance functions $r(x)$ which contains our case $r=r_{\infty}$.
Moreover, Slud (following Malevich ~\cite{ML} and Cuzick
~\cite{CZ}), established the central limit theorem for
$$\frac{Z_{\infty}(T)-\E Z_{\infty}(T)}{\sqrt{cT}}.$$

\end{proof}

\section{Proof of Theorem \ref{thm:main res}}
\label{sec:proof of CLT}

In this section we pursue the proof of the central limit theorem.
The main probabilistic tool we use is a result of Berk ~\cite{BR},
which establishes a central limit theorem for a triangular system of
random variables defined below. We start however with a general remark about
the zeros.

\subsection{A general remark about the distribution of the
zeros} \label{sec:phil rem}

Let $Y_{\infty}:\R\rightarrow\R$ be a Gaussian stationary process with
the covariance function $r_{\infty}(x) = \frac{\sin{x}}{x}.$ Such a
process exists, since
$r_{\infty}$ has a nonnegative Fourier transform on $\R$ (this is
related to Bochner's theorem, see e.g. ~\cite{CL}, page 126).
Moreover, we may assume with no loss of generality that $Y_{\infty}$
is almost surely everywhere differentiable. In fact, one may construct
$Y_{\infty}$ using its spectral representation. Alternatively, one
may use the Paley-Wiener construction
$Y(x)=\sum\limits_{n\in\Z}a_{n}\frac{\sin{(n-x)}}{n-x}$.

To define all the processes on the same space, we will assume that
$Y_{N}$ are defined on $\R$ by periodicity. We have the convergence
$r_{Y_{N}}(x)\rightarrow r_{Y_{\infty}}(x)$, as
$N\rightarrow\infty$. This implies the convergence of all the
finite-dimensional distributions of $Y_{N}$ to the
finite-dimensional distributions of $Y_{\infty}$. By theorem 9.2.2
~\cite{GS}, for
any continuous functional $\phi:C([a,b])\rightarrow \R$, the
distribution of $\phi(Y_{N})$ converges to the distribution of
$\phi(Y_{\infty})$ (one may easily check the
additional Lipshitz-like condition required by that theorem). 
Thus one could model a ``generic" statistic of
$X_{N}$ by the corresponding statistic of $Y_{\infty}$ on intervals,
growing linearly with $N$.

The convergence $r_{N}\rightarrow r_{\infty}$ suggests that the
distribution of the number of zeros of $X_{N}$ on the fixed interval
$[0,2\pi]$ is intimately related to the distribution of the number
of zeros of the fixed process $Y_{\infty}$ on growing intervals. The
particular case of the latter problem when the process is Gaussian
stationary (which is the case in this paper), has been studied
extensively over the past decades.

\subsection{Triangular systems of random variables}

Let $\tilde{\mathcal{L}} = \{ l_{k} \}$ be a (finite or infinite)
sequence of random variables (which we will refer as a {\em linear
system}), $K$ its length ($K=\infty$ if $\tilde{\mathcal{L}}$ is
infinite), and $\tilde{M}\ge 1$ an integer. We say that
$\tilde{\mathcal{L}}$ is $\tilde{M}$-dependent, if for every $i,j \ge 1$
with $i-j \ge \tilde{M}$, $\{l_{k}:\: k\le i \}$ and $\{l_{k}:\: k > j \}$
are independent. For example, a $0$-dependent linear
system $\tilde{\mathcal{L}}$ is independent.

One is usually interested in the distribution of the sums
$$S_{N}=\sum\limits_{k=1}^{N} l_{k}$$ as $N\rightarrow\infty$, that
is if $K=\infty$. In this situation one employs a special version of the CLT
due to Diananda ~\cite{DN}.
For a process $X(t)$ on $\R$ we denote $Z_{X}(T)$ to be the number
of zeros of $X$ on $[0,T]$. Cuzick ~\cite{CZ} employed Diananda's
result to prove the central limit theorem for $Z_{X}(T)$ as
$T\rightarrow\infty$ for a wide class of stationary processes $X$. A
more basic version of this theorem was applied earlier by Malevich
~\cite{ML} to obtain a similar, but more restrictive result.

Diananda's result applies to finite sums of random variables of
linear systems. The situation in our hands is somewhat different,
namely, of a so-called {\em triangular} system (or array) of random variables.
A triangular system of random variables is the correspondence
$K(N):\N\rightarrow\N\cup\{\infty\}$, together with a linear system
$$\tilde{\mathcal{L}}_{N} = \{ z_{N,k} : \: 1 \le k \le K(N)\}$$
for each $N\in\N$. We will use the notation $\tilde{\mathcal{Z}} =
\{ z_{N,k} \}$ to denote a triangular system.

Let $\tilde{M}=\tilde{M}(N)$ be sequence of integers. We say that
$\tilde{\mathcal{Z}}$ is $\tilde{M}$-dependent, if
$\tilde{\mathcal{L}}_{N}$ is $\tilde{M}(N)$-dependent for every $N$.

Given a triangular system $\tilde{\mathcal{Z}}$, we are usually
interested in the asymptotic distribution of the sums
$$S_{N}=\sum\limits_{k=1}^{K(N)} z_{N,k},$$ as $N\rightarrow\infty$. Note that here, unlike the linear systems,
both the number of the summands of $S_{N}$ and the summands
themselves depend on $N$. We have the following theorem due to Berk
~\cite{BR}, which establishes the asymptotic normality of $S_{N}$
for $\tilde{M}$-dependent triangular systems:

\begin{theorem}[Berk ~\cite{BR}]
\label{thm:Berk} Let $\mathcal{Z}=\{z_{N,k}: \: 1 \le k \le K(N)\}$
be a $\tilde{M}$-dependent triangular system of mean zero random
variables. Assume furthermore, that
\begin{enumerate}
\item
\label{it:Brk cond1} For some $\delta > 0$, $\E|z_{N,k}| ^{2+\delta}
\le A_{1}$, where $A_{1}>0$ is a  universal constant.

\item
\label{it:Brk cond2} For every $N$ and $1 \le i < j \le K(N)$, one
has
\begin{equation*}
\Var(z_{N,i+1}+\ldots+ z_{N,j}) \le (j-i)A_{2}
\end{equation*}
for some universal constant $A_{2}>0$.

\item
\label{it:Brk cond3} The limit
\begin{equation*}
\lim\limits_{N\rightarrow\infty} \frac{\Var(z_{N,1}+\ldots+
z_{N,K})}{K}
\end{equation*}
exists and is nonzero. Denote the limit $v>0$.

\item
\label{it:Brk cond4} We have
\begin{equation*}
\tilde{M} = o\bigg(K^{\frac{\delta}{2\delta+2}}\bigg).
\end{equation*}

\end{enumerate}

Then $$\frac{z_{N,1}+\ldots+ z_{N,K}}{\sqrt{vK}}$$ converges weakly
to $N(0,1)$.

\end{theorem}

Note that condition \ref{it:Brk cond4} requires, in particular, that
as $N\rightarrow\infty$, one has $K\rightarrow\infty$. Berk's result
was recently generalized by Romano and Wolf ~\cite{RW}.

\subsection{Plan of the proof of Theorem \ref{thm:main res}}
\label{sec:abt proof main res CLT}

We define the scaled processes
\begin{equation*}
Y_{N} (x) := X_{N} \bigg(\frac{x}{m}\bigg),
\end{equation*}
on $[0,2\pi m]$, where we reuse the notation $m=N+1/2$ from the
proof of Proposition \ref{prop:J asymp}. Let us denote their
covariance function
\begin{equation*}
r_{N}(x) =r_{Y_{N}}(x) = r_{X_{N}} \bigg(\frac{x}{m}\bigg) =
f_{N}(x),
\end{equation*}
with $f_{N}$ defined by \eqref{eq:f def}. It is obvious that
\begin{equation*}
Z_{X_{N}} = Z_{Y_{N}},
\end{equation*}
the number of zeros of $Y_{N}(x)$ on $$I'_{N}:=[0,2\pi m].$$ It will
be sometimes more convenient for us to work with $Y_{N}(x)$ and
$r_{N}(x)$ being defined (by periodicity) on
$$I_{N}:=[-\pi m, \pi m]$$ rather than on $I_{N}'$.

We are interested in the number $Z_{Y_{N}}$ of zeros of $Y_{N}$ on
intervals $I_{N}$, whose length grows linearly with $N$. Divide
$I_{N}$ into (roughly) $N$ subintervals $I_{N,k}$ of equal length
with disjoint interiors (so that the probability of having a zero on
an overlap is $0$), and represent $Z_{Y_{N}}$, almost surely, as a
sum of random variables $Z_{N,k}$, the number of zeros of $Y_{N}$ on
$I_{N,k}$. The stationarity of $Y_{N}$ implies that for a fixed $N$,
$Z_{N,k}$ are identically distributed (but by no means {\em
independent}).

We, therefore, obtain a triangular system
$$\tilde{\mathcal{Z}} = \{\tilde{Z}_{N,k}=Z_{N,k} - \E Z_{N,k} \}$$ of
mean zero random variables with growing rows. Just as $Z_{N,k}$, the
random variables $\tilde{Z}_{N,k}$ are identically distributed. We
will easily show that $\tilde{\mathcal{Z}}$ satisfies the conditions
\ref{it:Brk cond2}-\ref{it:Brk cond3} of Theorem \ref{thm:Berk}.
Moreover, we will see later that condition \ref{it:Brk cond1} holds
with $\delta=1$ (see Proposition \ref{prop:third mom ZM bnd}; here we
deal with a more complicated case of mollified random variables defined
below; an easier version of the same argument applies in case of
$\tilde{\mathcal{Z}}$, however it does not give the CLT due to the lack of
independence).

The main obstacle to this approach is that the random variables
$Z_{N,k}$ are {\em not independent} (and thus neither are
$\tilde{Z}_{N,k}$). In fact, we may give an explicit expression for
$$Cov(Z_{N,k_{1}},Z_{N,k_{2}})$$ in terms of an integral, which
involves $r_{N}$ and its derivatives. The stationarity of $Y_{N}$
implies that $Cov(Z_{N,k_{1}},Z_{N,k_{2}})$ depends on the
difference $k_{2}-k_{1}$ only (that is, the discrete process
$Z_{N,k}$ with $N$ fixed is stationary, provided that the continuous
process $Y_{N}$ is).

To overcome this obstacle, we notice that $r_{N}(x)$ and its couple
of derivatives are {\em small} outside a short interval around the
origin, and, moreover, their $L^{2}$ mass is concentrated around the
origin. This means that the dependencies between the values and the
derivatives of $Y_{N}(x)$ on $I_{N,k_{1}}$ and those on
$I_{N,k_{2}}$ are ``small" for $|k_{1}-k_{2}|$ sufficiently large.
Thus the system $\tilde{\mathcal{Z}}$ is ``almost $M$-independent",
provided that $M=M(N)$ is large enough (it is sufficient to take any
sequence $M(N)$ growing to infinity; see Proposition
\ref{prop:E(Z-Zmol)^2=o(N)}).

One may then hope to exchange the process $Y_{N}$ (and thus the
system $\tilde{\mathcal{Z}}$) with a process $Y_{N}^{M}$ (resp.
$\tilde{\mathcal{Z}}^{M}$), where $M=M(N)$ is a growing parameter,
so that the distributions of the number of zeros $Z_{N}=Z_{Y_{N}}$
and $Z_{N}^{M}=Z_{Y_{N}^{M}}$ of $Y_{N}$ and $Y_{N}^{M}$
respectively, are asymptotically equivalent, and the above
properties of the original system stay unimpaired. In addition, we
require $\tilde{\mathcal{Z}}^{M}$ to be $M$-dependent (or, rather,
$const\cdot M$-dependent). To prove the asymptotic equivalence of
$Z_{N}$ and $Z_{N}^{M}$ we will evaluate the variance of the
difference $\Var(Z_{N}-Z_{N}^{M})$ (see Proposition
\ref{prop:E(Z-Zmol)^2=o(N)}).

To define $Y_{N}^{M}$, we introduce a function $r_{N}^{M} =
r_{N}\cdot S_{M}$, where $|S_{M}|\le 1$ is a sufficiently smooth
function supported on $[-const\cdot M, const\cdot M]$ approximating
the unity near the origin with a positive Fourier transform on the
circle $I_{N}$ (with end-points identified). We require that $r_{N}^{M}$ (and a couple of its
derivatives) preserve $100\%$ of the $L^{2}$ mass of $r_{N}$ (resp.
a couple of its derivatives) (see Lemma \ref{lem:|rNM-rN|L2->0}). We
then construct $Y_{N}^{M}$ with covariance function $r_{N}^{M}$
using some Fourier analysis on $I_{N}$. It is important to observe
that the covariance function being supported essentially at $[-M,M]$
means that $\tilde{\mathcal{Z}}^{M}$ is (roughly) $M$-independent,
in the periodic sense.

To get rid of the long-range dependencies, which occur as a result
of the periodicity of $Y_{N}$, we prove the central limit theorem
for positive and negative zeros separately (see the proof of
Proposition \ref{prop:CLT for Zmol}). Namely we define $Z^{M,+}$
(resp. $Z^{M,-}$) to be the number of zeros $z$ of $Y_{N}^{M}$ with
$z>0$ (resp. $z<0$), and
$$Z^{M}=Z^{M,+}+Z^{M,-}$$ almost surely. We are going to prove the
asymptotic normality of the distribution of $Z^{M,+}$ and similarly,
of $Z^{M,-}$. We will prove that this will imply the asymptotic
normality of the sum $Z^{M}$.

Concerning the choice of $M=M(N)$, on one hand, to well approximate
$Y_{N}$, $M$ has to grow to infinity with $N$. On the other hand,
condition \ref{it:Brk cond4} of theorem \ref{thm:Berk} constrains
the growth rate of $M$ from above. The above considerations leave us
a handy margin for $M$.

\subsection{Some conventions}
In this section we will use some Fourier analysis with functions
defined on the circle $I_{N}:=[-\pi m, \pi m]$ (or equivalently,
$I_{N}':=[0,2\pi m]$). We will adapt the following conventions. Let
$f:I_{N}\rightarrow\R$ be a real-valued function. For $n\in\Z$, we
define
\begin{equation*}
\hat{f}(n) = \int\limits_{I_{N}} f(x) e^{-\frac{in}{m}x}
\frac{dx}{\sqrt{2\pi m}}.
\end{equation*}

\vspace{5mm}

If $f$ is a real valued {\em even} nice function, then
\begin{equation*}
r(x) = \hat{r}(0)\cdot\frac{1}{\sqrt{2\pi m}} +
\sum\limits_{n=1}^{\infty} \sqrt{2}\hat{r}(n)\cdot
\frac{\cos(\frac{nx}{m})}{\sqrt{\pi m}}.
\end{equation*}
With the above conventions, if $f,g:I_{N}\rightarrow\R$ are two
functions, then
\begin{equation*}
\hat{(f\cdot g)} (n) = \frac{1}{\sqrt{2\pi m}} (\hat{f}*\hat{g})(n),
\end{equation*}
and
\begin{equation*}
(\hat{f*g})(n) = \sqrt{2\pi m}\hat{f}(n)\cdot \hat{g}(n).
\end{equation*}

For the real valued even functions, the Parseval identity is
\begin{equation*}
\|f \|_{L^{2}(I_{N})}^{2} = \|\hat{f}\|_{l^{2}(\Z)}^{2} =
\hat{f}(0)^2+\sum\limits_{n=1}^{\infty} 2 \hat{f}(n)^2.
\end{equation*}

\subsection{Proof of Theorem \ref{thm:main res}}

\begin{proof}[Proof of Theorem \ref{thm:main res}]

Let $Y_{N}(x)=X_{N}(\frac{x}{m})$, and for notational convenience,
we assume by periodicity, that $Y_{N}$ and its covariance function
$r_{N}$ are defined on $I_{N}:=[-\pi m, \pi m]$. One may rewrite the
definition of $Y_{N}$ using
\begin{equation}
\label{eq:rN Fourier repr} r_{N}(x) = \hat{r}_{N}(0)\cdot
\frac{1}{\sqrt{2\pi m}} + \sum\limits_{n=1}^{\infty}
\sqrt{2}\hat{r}_{N}(n) \cdot \frac{1}{\sqrt{\pi
m}}\cos\bigg(\frac{n}{m}x\bigg),
\end{equation}
as
\begin{equation}
\label{eq:spect form YN}
\begin{split}
 Y_{N}(x) &= \sqrt{\hat{r}_{N}(0)}\frac{1}{(2\pi
m)^{1/4}} a_{0} \\&+ \sum\limits_{n=1}^{\infty}
2^{1/4}\sqrt{\hat{r}_{N}(n)} \cdot \frac{1}{(\pi m)^{1/4}}
\bigg(a_{n}\cos\bigg(\frac{n}{m}x\bigg)+b_{n}\sin\bigg(\frac{n}{m}x\bigg)\bigg),
\end{split}
\end{equation}
where $a_{n}$ and $b_{n}$ are $(0,1)$ Gaussian i.i.d. One may
compute $\hat{r}_{N}$ to be
\begin{equation*}
\hat{r}_{N}(n) = \begin{cases} \frac{\sqrt{\pi m}}{\sqrt{2} N}, \: &1\le |n| \le N \\
0, &\text{otherwise}
\end{cases} = \frac{\sqrt{\pi m}}{\sqrt{2} N} \chi_{1\le |n| \le N} (n).
\end{equation*}

It is easy to identify \eqref{eq:spect form YN} as the spectral form
of $Y_{N}$ on the circle, analogous to the well-known spectral
theory on the real line (see e.g. ~\cite{CL}, section 7.5). The
spectral representation proved itself as extremely useful and
powerful while studying various properties of stationary processes.

Let $0 < M < \pi m$ be a large parameter and $\chi_{[-M,M]}$ be the
characteristic function of $[-M,M]\subseteq I_{N}$. Define
\begin{equation*}
S_{M}(x) = \frac{(\chi_{[-M,M]})^{*8} (x)}{C M^7},
\end{equation*}
where $(\cdot) ^{*l}$ stands for convolving a function $l$ times to
itself, and the universal constant $C>0$ is chosen so that
$S_{M}(0)=1$. The function $S_{M}:I_{N}\rightarrow\R$ is a piecewise
polynomial of degree $7$ in $\frac{|x|}{M}$, independent of $M$. It
is a $6$-times continuously differentiable function supported at
$[-8M,8M]$. For $|x|< 2M$, for example,
\begin{equation}
\label{eq:S_M=1+...} S_{M}(x) =
1+b_{1}\bigg(\frac{x}{M}\bigg)^2+b_{2}\bigg(\frac{x}{M}\bigg)^4+b_{3}\bigg(\frac{x}{M}\bigg)^6+
b_{4}\bigg(\frac{|x|}{M}\bigg)^7
\end{equation}
for some constants $b_{1},\ldots,b_{4}\in\R$, which may be easily
computed.

We define the mollified covariance function
$r^{M}=r_{N}^{M}:I_{N}\rightarrow\R$ by
\begin{equation}
\label{eq:rNM def}
r_{N}^{M}(x) := r_{N}(x) \cdot S_{M}(x),
\end{equation}
with the Fourier series given by
\begin{equation}
\label{eq:rM Fourier trans} \hat{r}_{N}^{M}(n) = \frac{1}{\sqrt{2\pi
m}} \cdot (\hat{r}_{N} * \hat{S}_{M})(n) = \frac{1}{2 N} (\chi_{1\le
|n| \le N} * \hat{S}_{M}) (n) \ge 0,
\end{equation}
since
\begin{equation}
\label{eq:SM Fourier transf comp} \hat{S}_{M}(n) =  \frac{(2\pi
m)^{7/2}}{CM^7} \cdot (\hat{\chi}_{[-M,M]} (n))^8 \ge 0.
\end{equation}
One may compute explicitly the Fourier transform of $\chi_{[-M,M]}$
to be
\begin{equation}
\label{eq:chi[-M,M] Fourier transf}
\hat{\chi}_{[-M,M]} (n) = \begin{cases} \sqrt{\frac{2}{\pi}}\frac{M}{\sqrt{m}}, &n=0 \\
\sqrt{\frac{2}{\pi}} \cdot \frac{\sqrt{m}}{n} \sin(\frac{nM}{m}),
&n\ne 0
\end{cases}.
\end{equation}

The nonnegativity of $\hat{S}_{M}$ allows us to construct a process
$Y_{N}^{M}(x)$ on $I_{N}$ with covariance function $r_{Y_{N}^{M}} =
r_{N}^{M}$ as
\begin{equation}
\label{eq:spect form YNM}
\begin{split}
Y_{N}^{M}(x) &= \sqrt{\hat{r}_{N}^{M}(0)}\frac{1}{(2\pi m)^{1/4}}
a_{0}\\&+\sum\limits_{n=1}^{\infty} \sqrt{\hat{r}_{N}^{M}(n)} \cdot
\frac{2^{1/4}}{(\pi m)^{1/4}} \bigg(a_{n}
\cos\bigg(\frac{n}{m}x\bigg) + b_{n}
\sin\bigg(\frac{n}{m}x\bigg)\bigg),
\end{split}
\end{equation}
the RHS being almost surely an absolutely convergent series,
uniformly w.r.t. $x$.

\begin{remark}
We assume that the $a_{n}$ and $b_{n}$ for $n\le N$
in \eqref{eq:spect form YNM} are the same as in \eqref{eq:spect form YN},
so that $Y_{N}^{M}(x) $ converges in $L^2$ to $Y_{N}(x)$ (see Lemma
\ref{eq:YNM->YN L2}).
\end{remark}

Let $M=M(N)$ be any sequence of numbers growing to infinity,
satisfying $M=o(N^{1/4})$. Proposition \ref{prop:CLT for Zmol} then
implies that as $N\rightarrow\infty$, $\frac{Z_{N}^{M} - \E
Z_{N}^{M}}{\sqrt{c N}}$ is asymptotically normal. Proposition
\ref{prop:E(Z-Zmol)^2=o(N)} states that
$$\Var(Z_{N}^{M} - Z_{N}) = o(\Var(Z_{N})),$$ so that the
distribution of $\frac{Z_{N} - \E Z_{N}}{\sqrt{c N}}$ is
asymptotically equivalent to that of $\frac{Z_{N}^{M} - \E
Z_{N}^{M}}{\sqrt{c N}}$, which implies the statement of Theorem
\ref{thm:main res}.

\end{proof}

\begin{proposition}
\label{prop:E(Z-Zmol)^2=o(N)} Suppose that as $N\rightarrow\infty$,
we have $M\rightarrow\infty$. Then one has
\begin{equation*}
\Var (Z_{N}-Z_{N}^{M}) = o(N).
\end{equation*}

The proof of Proposition \ref{prop:E(Z-Zmol)^2=o(N)} is given in
section \ref{sec:pr E(Z-Zmol)^2=o(N)}.

\end{proposition}

\subsection{Proof of CLT for $Z_{N}^{M}$}

The main result of the present section is Proposition \ref{prop:CLT
for Zmol}, which establishes the central limit theorem for the
mollified random variable $Z^{M}$.

\begin{proposition}
\label{prop:CLT for Zmol} Suppose that as $N\rightarrow\infty$,
$M=o(N^{1/4})$. Then for $N\rightarrow\infty$, the random variables
$\frac{Z_{N}^{M}-\E Z_{N}^{M}}{\sqrt{cN}}$, weakly converge to the
standard Gaussian $N(0,1)$.
\end{proposition}

\begin{proof}[Proof of Proposition \ref{prop:CLT for Zmol}]
Recall that $Z_{N}^{M}$ is the number of zeros of $Y_{N}^{M}$ on
$I_{N}=[-\pi m, \pi m]$. We are going to prove a central limit
theorem for the number of zeros on $I_{N}^{+}:=[0, \pi m]$ (denote
similarly $I_{N}^{-}:=[-\pi m, 0]$) only. We thereby denote
$Z_{N}^{M,+}$ (resp. $Z_{N}^{M,-}$) the number of zeros of
$Y_{N}^{M}$ on $I_{N}^{+}$ (resp. $I_{N}^{-}$), and analogously,
$Z_{N}^{+}$ and $Z_{N}^{-}$ will denote the number of zeros of
$Y_{N}$ on $I_{N}^{+}$ and $I_{N}^{-}$ respectively. This also
implies a central limit theorem for $Z_{N}^{M,-}$ by the
stationarity. The problem is that $Z_{N}^{M,+}$ and $Z_{N}^{M,-}$
are not independent, so that writing $Z_{N}^{M} = Z_{N}^{M,+}
+Z_{N}^{M,-}$ a.a. does not imply the asymptotic normality for the
sum. Therefore we will have to come up with a more gentle argument
in the end of this proof.

%We could just prove the CLT for the zeros on [8M,-2\pi m- 8M], applying Berk directly.

Let $L>0$ be any integer, which we will keep fixed throughout the
proof. We divide $I_{N}^{+}$ into subintervals
$$I_{N,k}= \bigg[ (k-1)\cdot\frac{\pi m}{LN}, k\cdot\frac{\pi
m}{LN}\bigg]$$ for $1 \le k \le L N$, and denote $Z_{N,k}^{M}$ the
number of zeros of $Y_{N}^{M}(x)$ on $I_{N,k}$.

Recall that, as a function on $[-\pi m, \pi m]$, $r_{N}^{M}$ is
supported on $$[-8 \cdot M, 8 \cdot M].$$ Therefore, if $i-j
\ge 8 L M $, the random variables $\{Z_{N,k}^{M}:\: k \le i\}$
are independent of $\{ Z_{N,k}^{M}:\:  k > j\}$.

We apply Theorem \ref{thm:Berk} on the $\tilde{M}=const \cdot
M$-dependent triangular system $$\tilde{\mathcal{Z}}^{M} =
\{\tilde{Z}_{N,k}^{M} := Z_{N,k}^{M}-\E Z_{N,k}^{M}:\: 1 \le k \le
K(N) \}$$ with $K(N)=NL$. Since with probability $1$ neither of
$Y_{N}^{M}$ have zeros on the overlaps of $I_{N,k}$, we have
\begin{equation*}
Z_{N}^{M,+}-\E Z_{N}^{M,+} = \sum\limits_{k=1}^{NL}
\tilde{Z}_{N,k}^{M}
\end{equation*}
almost surely, so that to finish the proof of this Proposition, it
remains to check that $\tilde{\mathcal{Z}}^{M}$ satisfies conditions
\ref{it:Brk cond1}-\ref{it:Brk cond4} of Theorem \ref{thm:Berk}.

Proposition \ref{prop:third mom ZM bnd} implies that condition
\ref{it:Brk cond1} is satisfied with $\delta = 1$, provided that we
choose $L$ large enough. Since $\tilde{M} \sim const \cdot M$ and
$K(N) \sim const\cdot N$, the assumption $M=o(N^{1/4})$ of the
present Proposition is equivalent to condition \ref{it:Brk cond4}.

Condition \ref{it:Brk cond3} of Theorem \ref{thm:Berk} is equivalent
to $\Var(Z_{N}^{M,+}) \sim c_{1} N$ for some $c_{1} > 0$. An
application of \eqref{eq:var Z sim} together with Proposition
\ref{prop:E(Z-Zmol)^2=o(N)} and the triangle inequality, imply that
\begin{equation*}
\Var(Z_{N}^{M}) \sim c N.
\end{equation*}
One may also derive the corresponding result for $Z_{N}^{M,+}$,
starting from $\Var(Z_{N}^{+}) \sim \frac{c}{2} N$ (the proof
follows along the same lines as the proof of \eqref{eq:var Z sim})
and using \eqref{eq:E(Z+-ZM+)^2} with the triangle inequality.

It then remains to check that $ \tilde{\mathcal{Z}}^{M}$ satisfies
condition \ref{it:Brk cond2} of Theorem \ref{thm:Berk}. Using the
same approach we used in the course of the proof of \eqref{eq:var Z
sim}, one may find out that
\begin{equation}
\label{eq:bnd par sum var}
\begin{split}
&\Var(\tilde{Z}_{N,i+1}^{M}+\ldots + \tilde{Z}_{N,j}^{M}) \\&=
\frac{2}{\pi^2}\int\limits_{0} ^{(j-i)\frac{\pi m}{LN}} \bigg[
\bigg((j-i)\frac{\pi m}{LN} -x\bigg) \cdot \bigg(
\frac{{\lambda_{2,N}^{M}}'(1-{r(x)}^2)-{r'(x)}^2}{{(1-{r(x)}^2)}^{3/2}}\big(\sqrt{1-\rho(x)^2}\\&+\rho(x)\arcsin{\rho(x)}\big)
- {\lambda_{2,N}^{M}}' \bigg) \bigg] dx + (j-i)\frac{m}{LN}
\sqrt{{\lambda_{2,N}^{M}}'},
\end{split}
\end{equation}
where we use the shortcuts $r=r_{N}^{M}$, ${\lambda_{2,N}^{M}}' =
-{r_{N}^{M}}''(0)$, and
\begin{equation*}
\rho(x) = \rho_{N}^{M}(x) = \frac{r''(x)(1-r(x)^2)+r'(x)^2
r(x)}{\lambda_{2}'(1-r(x)^2)-r'(x)^2}.
\end{equation*}
We have
\begin{equation*}
(j-i)\frac{m}{LN} \sqrt{{\lambda_{2,N}^{M}}'} \ll (j-i),
\end{equation*}
since $\frac{m}{N} \le 2$ and ${\lambda_{2,N}^{M}}' = O(1)$. It
remains therefore to bound the integral in \eqref{eq:bnd par sum
var}, which we denote $J$. We write, denoting $\tau :=
(j-i)\frac{\pi m}{LN}$:
\begin{equation}
\label{eq:bnd par sum var int man} J \ll (j-i)\int\limits_{0}^{\tau}
\bigg[
\frac{{\lambda_{2}^{M}}'(1-{r}^2)-{r'}^2}{{(1-{r}^2)}^{3/2}}\big(\sqrt{1-\rho^2}+\rho\arcsin{\rho}\big)
- {\lambda_{2}^{M}}' \bigg]dx.
\end{equation}

It will suffice then to prove that the latter integral is uniformly
bounded. Let $K_{N}^{M}(x)$ be the integrand. Expanding
$K_{N}^{M}(x)$ into Taylor polynomial around the origin, as we did
in the course of the proof of \eqref{eq:var Z sim} (see the proof of
Lemma \ref{lem:int est t small}), we find that $K_{N}^{M}(x)$ is
uniformly bounded on some fixed neighbourhood of the origin (say, on
$[0,c]$). We claim, that outside $[0,c]$, the integrand is rapidly
decaying, uniformly with $N$.

It is easy to see that $r(x)$ being supported at $[0,const\cdot M]$
implies that $K(x)$ is supported in $[0,const\cdot M]$ as well (note
that we exploit here the fact that by counting only the positive zeros we disregard the dependencies
between zeros near $-\pi m$ and $\pi m$).
Moreover, on $[c,const\cdot M]$, $|K_{N}^{M}(x)| \ll
\frac{1}{x^2}$, where the constant involved in the ``$\ll$"-notation
is universal. Therefore the integral on the RHS of \eqref{eq:bnd par
sum var int man} is uniformly bounded, so that $J\ll (j-i)$, which
verifies condition \ref{it:Brk cond2} of Berk's theorem.

This concludes the proof of the asymptotic normality for
$Z_{N}^{M,+}$ (and also $Z_{N}^{M,-}$). Having that result in our
hands, we define the random variables $\hat{Z}_{N}^{M,+}$ and
$\hat{Z}_{N}^{M,-}$ to be the number of zeros of $Y_{N}^{M}$ on
$[8M,\pi m -8 M]$ and $[-\pi m+8M,-8 M]$ respectively. The random
variables $\hat{Z}_{N}^{M,\pm}$ are {\em independent}, since
$r_{N}^{M}$ is supported on $[-8M,8M]$.

In addition, let $Z_{N,S}^{M,+}$, $Z_{N,L}^{M,+}$, $Z_{N,S}^{M,-}$
and $Z_{N,L}^{M,-}$ be the number of zeros of $Y_{N}^{M}$ on
$[0,8M]$, $[\pi m -8M, \pi m]$, $[-8M,0]$ and $[-\pi m,-\pi m +8M]$
respectively. We have
\begin{equation}
\begin{split}
\label{eq:Var small int} &\Var Z_{N,S}^{M,+},\, \Var
Z_{N,L}^{M,+},\, \Var Z_{N,S}^{M,-},\, \Var Z_{N,L}^{M,-} \ll M \\&=
o(\Var Z_{N}^{M,+}), \, o(\Var Z_{N}^{M,-}),\, o(\Var Z_{N}^{M})
\end{split}
\end{equation}
by condition \eqref{it:Brk cond2} of Theorem \ref{thm:Berk} (which
we validated). Therefore $$\hat{Z}_{N}^{M,+} = Z_{N}^{M,+}-
Z_{N,S}^{M,+}-Z_{N,L}^{M,+}$$ and $$\hat{Z}_{N}^{M,-} = Z_{N}^{M,-}-
Z_{N,S}^{M,-}-Z_{N,L}^{M,-}$$ converge to the Gaussian distribution.

%For example using the characteristic function sum of variables -> product of char. fnc.
The independence of $\hat{Z}_{N}^{M,\pm}$ then implies the
asymptotic normality of $\hat{Z}_{N}^{M,+}+\hat{Z}_{N}^{M,-}$, and
finally we obtain the asymptotic normality of
\begin{equation*}
Z_{N}^{M}=(\hat{Z}_{N}^{M,+}+\hat{Z}_{N}^{M,-})+ Z_{N,S}^{M,+}+
Z_{N,L}^{M,+}+  Z_{N,S}^{M,-}+ Z_{N,L}^{M,-},
\end{equation*}
again by \eqref{eq:Var small int}.

\end{proof}

\section{Proof of Proposition \ref{prop:E(Z-Zmol)^2=o(N)}}
\label{sec:pr E(Z-Zmol)^2=o(N)}

\subsection{Introduction and the basic setting}
\label{sec:prop Z-Zmol into}

Recall that we have the processes $Y_{N}(x)$ and $Y_{N}^{M}(x)$,
defined on $I_{N}=[-\pi m,\pi m]$ and are interested in the
distribution of $Z_{N}$ and $Z_{N}^{M}$, the number of zeros of
$Y_{N}$ and $Y_{N}^{M}$ on $I_{N}$ respectively. The goal of the
present section is to prove the bound
\begin{equation}
\label{eq:E(Z-ZM)^2} \Var (Z_{N}-Z_{N}^{M})=o(N)
\end{equation}
on the variance of the difference. For notational convenience, we
will consider only the positive zeros, that is, let $Z_{N}^{+}$
(resp. $Z_{N}^{M,+}$) be the number of zeros of $Y_{N}$ (resp.
$Y_{N}^{M}$) on $I_{N}^{+}=[0,\pi m]$. We will prove that
\begin{equation}
\label{eq:E(Z+-ZM+)^2} \Var (Z_{N}^{+}-Z_{N}^{M,+}) = o(N),
\end{equation}
and by the stationarity, it will also imply
\begin{equation}
\label{eq:E(Z--ZM-)^2} \Var (Z_{N}^{-}-Z_{N}^{M,-}) = o(N),
\end{equation}
where we denoted the number of negative zeros in an analogous
manner. Finally, \eqref{eq:E(Z+-ZM+)^2} together with
\eqref{eq:E(Z--ZM-)^2}, will imply \eqref{eq:E(Z-ZM)^2}, by the
triangle inequality.

Let $S>0$ and $R>0$ be a couple of large integral parameters. We
divide $I_{N}^{+}$ into $K=2 S m$ equal subintervals, so that
$$I_{N,k}=\bigg[(k-1)\frac{2\pi m}{K},k\frac{2\pi m}{K}\bigg] $$ for
$1\le k \le K$.

We then write the LHS of \eqref{eq:E(Z+-ZM+)^2} as
\begin{equation}
\label{eq:Var ZN+-ZNM+ = sum cov} \Var(Z_{N}^{+}-Z_{N}^{M,+}) =
\sum\limits_{k_{1},k_{2}=1}^{K} Cov \big(
Z_{N,k_{1}}-Z_{N,k_{1}}^{M}, Z_{N,k_{2}}-Z_{N,k_{2}}^{M}\big).
\end{equation}
We divide the last summation into $3$ different ranges. That is, we
define
\begin{equation}
\label{eq:E1 def} E_{1} = \sum\limits_{|k_{1}-k_{2}|\le 1},
\end{equation}
\begin{equation}
\label{eq:E2 def} E_{2} = \sum\limits_{2\le |k_{1}-k_{2}|\le SR},
\end{equation}
and
\begin{equation}
\label{eq:E3 def} E_{3} = \sum\limits_{|k_{1}-k_{2}|\ge SR}^{K},
\end{equation}
and prove that for $1\le i \le 3$
$$\lim\limits_{N\rightarrow\infty} \frac{E_{i}}{N} = 0.$$

\subsection{Preliminaries}
\label{sec:prob prelim cov}

In addition to the covariance functions $r=r_{N}$ and
$r^{M}=r_{N}^{M}$ of $Y_{N}$ and $Y_{N}^{M}$ respectively, defined
on $I_{N}$, we introduce the joint covariance function
\begin{equation}
\label{eq:rNM0 def} r^{M,0}(x)=r_{N}^{M,0}(x) = \E\big[
Y_{N}(y)Y_{N}^{M}(y+x)\big],
\end{equation}
which is a function of $x$ indeed, by stationarity. Similarly to
\eqref{eq:|r|<=1}, one has $|r^{M,0}|\le 1$ by the Cauchy-Schwartz
inequality.

Using the spectral form \eqref{eq:spect form YN} (resp.
\eqref{eq:spect form YNM}) of $Y_{N}$ (resp. $Y_{N}^{M}$), one may
compute the Fourier expansion of $r_{N}^{M,0}$ to be
\begin{equation}
\label{eq:rNM0 Fourier repr}
\begin{split}
r_{N}^{M,0}(x) &= \frac{\sqrt{\hat{r}(0)\cdot
\hat{r}^{M}(0)}}{\sqrt{2\pi m}} + \sum\limits_{n=1}^{\infty}
\sqrt{\hat{r}(n)\hat{r}^{M}(n)} \cdot \frac{\sqrt{2}}{\sqrt{\pi m}}
\cos\bigg(\frac{n}{m}x\bigg) \\&= \sum\limits_{n=1}^{\infty}
\sqrt{\hat{r}(n)\hat{r}^{M}(n)} \cdot \frac{\sqrt{2}}{\sqrt{\pi m}}
\cos\bigg(\frac{n}{m}x\bigg) \\&= \frac{1}{N}\sum\limits_{n=1}^{N}
\sqrt{(\chi_{1\le |n| \le N} * \hat{S}_{M})(n)}\cdot \frac{1}{(2\pi
m)^{1/4}} \cos\bigg(\frac{n}{m}x\bigg).
\end{split}
\end{equation}
In particular, $r_{N}^{M,0}$ is {\em even}, and
\begin{equation}
\label{eq:hatrNM0} \hat{r}_{N}^{M,0}(n) = \sqrt{\hat{r}_{N}(n)
\hat{r}_{N}^{M}(n)} = \bigg(\frac{\pi m}{8}\bigg)^{1/4}
\frac{1}{N}\cdot \chi_{1\le |n| \le N}(n) \cdot \sqrt{(\chi_{1\le
|n| \le N} * \hat{S}_{M})(n)}.
\end{equation}

Recall that to determine the second moment $\E Z_{X}^2$ of a process
$X$, we naturally encountered the distribution of the random vector
\eqref{eq:vec val der t1t2}. Similarly, to evaluate the covariances
in \eqref{eq:Var ZN+-ZNM+ = sum cov}, one naturally encounters the
distribution of vectors
$$W_{1}=\big(Y_{N}^{M}(x_{1}),Y_{N}^{M}(x_{2}),{Y_{N}^{M}}'(x_{1}),{Y_{N}^{M}}'(x_{2})\big)$$
with probability density
$\phi_{N,M}^{x_{1},x_{2}}(u_{1},u_{2},v_{1},v_{2})$ and
$$W_{2}=\big(Y_{N}(x_{1}),Y_{N}^{M}(x_{2}),Y_{N}'(x_{1}),{Y_{N}^{M}}'(x_{2})\big)$$
with probability density
$\phi_{N,M,0}^{x_{1},x_{2}}(u_{1},u_{2},v_{1},v_{2})$ for some
$x_{1},x_{2}\in I_{N}$. As before, the distributions
$\phi_{N,M}^{x_{1},x_{2}}$ and $\phi_{N,M,0}^{x_{1},x_{2}}$ depend
only on $x=x_{2}-x_{1}$ by stationarity, and we will denote
$\phi_{N,M}^{x}=\phi_{N,M}^{x_{1},x_{2}}$ and
$\phi_{N,M,0}^{x}=\phi_{N,M,0}^{x_{1},x_{2}}$ accordingly.

Similarly to the mean zero Gaussian distribution with covariance
matrix \eqref{eq:cov mat} of the random vector \eqref{eq:vec val der
t1t2}, both $W_{1}$ and $W_{2}$ are mean zero Gaussian with
covariance matrices
\begin{equation}
\label{eq:Sigma NM def}
\Sigma_{N,M}(x) = \left(\begin{matrix}  1 &r_{N}^{M}(x) &0 &{r_{N}^{M}}'(x) \\
r_{N}^{M}(x) &1 &-{r_{N}^{M}}'(x) &0 \\ 0 &-{r_{N}^{M}}'(x)
&{\lambda_{2,N}^{M}}' &-{r_{N}^{M}}''(x)
\\ {r_{N}^{M}}'(x) &0 &-{r_{N}^{M}}''(x) &{\lambda_{2,N}^{M}}'
\end{matrix} \right)
\end{equation}
and
\begin{equation}
\label{eq:Sigma NM0 def}
\Sigma_{N,M,0}(x) = \left(\begin{matrix} 1 &r_{N}^{M,0}(x) &0 &{r_{N}^{M,0}}'(x) \\ r_{N}^{M,0}(x) &1 &-{r_{N}^{M,0}}'(x) &0 \\
0 &-{r_{N}^{M,0}}'(x) &\lambda_{2,N}' &-{r_{N}^{M,0}}''(x) \\
{r_{N}^{M,0}}'(x) &0 &-{r_{N}^{M,0}}''(x) & {\lambda_{2,N}^{M}}'
\end{matrix} \right),
\end{equation}
where, as usual, we denote
\begin{equation*}
\lambda_{2,N}' := -r_{N}''(0);\; {\lambda_{2,N}^{M}}' :=
{-r_{N}^{M}}''(0).
\end{equation*}

Similarly to $\Sigma(t)$ in \eqref{eq:cov mat}, both
$\Sigma_{N,M}(x)$ and $\Sigma_{N,M,0}(x)$ are nonsingular for $x\ne
0$, and so
\begin{equation}
\label{eq:phi NM def} \phi_{N,M}^{x}(w) =  \frac{1}{(2\pi)^2
\sqrt{\det{\Sigma_{N,M}(x)}}}e^{-\frac{1}{2}
w\Sigma_{N,M}(x)^{-1}w^t}
\end{equation}
and
\begin{equation}
\label{eq:phi NM0 def} \phi_{N,M,0}^{x}(w) = \frac{1}{(2\pi)^2
\sqrt{\det{\Sigma_{N,M,0}(x)}}}e^{-\frac{1}{2}
w\Sigma_{N,M,0}(x)^{-1}w^t}.
\end{equation}

We denote
\begin{equation}
\label{eq:tildphiNM,NM0 def} \tilde{\phi}_{N,M}^{x}(v_{1},v_{2}) :=
\phi_{N,M}^{x}(0,0,v_{1},v_{2}) ;\quad
\tilde{\phi}_{N,M,0}^{x}(v_{1},v_{2}) :=
\phi_{N,M,0}^{x}(0,0,v_{1},v_{2})
\end{equation}
and define the random vector
\begin{equation*}
\big(V_{1}=V_{1}(x),V_{2}=V_{2}(x)\big) = (Y_{N}'(0),
{Y_{N}^{M}}'(x))
\end{equation*}
conditioned upon $Y_{N}(0)={Y_{N}^{M}}(x)=0$ with probability
density function $\psi_{N,M,0}^{x}(v_{1},v_{2})$. The random vector
$(V_{1},V_{2})$ has a mean zero Gaussian distribution with
covariance matrix
\begin{equation}
\label{eq:OmegaNM0 cond def} \Omega_{N,M,0}^{x} =
\left(\begin{matrix}\lambda_{2,N}' -
\frac{{r_{N}^{M,0}}'(x)^2}{1-r_{N}^{M,0}(x)^2} & {-r_{N}^{M,0}}''(x)
-\frac{r_{N}^{M,0}(x)\cdot {r_{N}^{M,0}}'(x)^2 }{1-r_{N}^{M,0}(x)^2} \\
{-r_{N}^{M,0}}''(x) -\frac{r_{N}^{M,0}(x)\cdot {r_{N}^{M,0}}'(x)^2
}{1-r_{N}^{M,0}(x)^2} & {\lambda_{2,N}^{M}}'-
\frac{{r_{N}^{M,0}}'(x)^2}{1-r_{N}^{M,0}(x)^2}\end{matrix} \right),
\end{equation}
The matrix $\Omega_{N,M,0}^{x}$ is regular for $x\ne 0$. We have,
analogously to \eqref{eq:phi=psi/det}
\begin{equation*}
\tilde{\phi}_{N,M,0}^{x}(v_{1},v_{2}) =
\frac{\psi_{N,M,0}^{x}(v_{1},v_{2})}{2\pi
\sqrt{1-r_{N}^{M,0}(x)^2}}.
\end{equation*}

Similarly, let $\psi_{N,M}^{x}(v_{1},v_{2})$ be the probability
density function of $({Y_{N}^{M}}'(0),{Y_{N}^{M}}'(x))$ conditioned
upon ${Y_{N}^{M}}(0)={Y_{N}^{M}}(x)=0$. One then has
\begin{equation}
\label{eq:phiNM=psiNM/det} \tilde{\phi}_{N,M}^{x}(v_{1},v_{2}) =
\frac{\psi_{N,M}^{x}(v_{1},v_{2})}{2\pi \sqrt{1-r_{N}^{M}(x)^2}}.
\end{equation}

\subsection{Auxiliary Lemmas}

\begin{lemma}
\label{lem:|rNM-rN|L2->0} One has the following estimates
\begin{enumerate}

\item \label{it:rNM-rN L2 norm}
\begin{equation*}
\|r_{N}^{M} -r_{N}\|_{L^{2}(I_{N})} =
O\bigg(\frac{1}{\sqrt{M}}\bigg),
\end{equation*}

\item \label{it:rNM0-rN L2 norm}
\begin{equation*}
\|r_{N}^{M,0} -r_{N}\|_{L^{2}(I_{N})} =
O\bigg(\frac{1}{M^{1/4}}\bigg),
\end{equation*}

\item \label{it:rNM''-rN'' L2 norm}
\begin{equation*}
\|{r_{N}^{M}}'' -r_{N}''\|_{L^{2}(I_{N})} =
O\bigg(\frac{1}{\sqrt{M}}\bigg),
\end{equation*}

\item \label{it:rNM0''-rN'' L2 norm}
\begin{equation*}
\|{r_{N}^{M,0}}'' -r_{N}''\|_{L^{2}(I_{N})} =
O\bigg(\frac{1}{M^{1/4}}\bigg).
\end{equation*}

\item \label{it:rNM'-rN' L2 norm}
\begin{equation*}
\|{r_{N}^{M}}' -r_{N}'\|_{L^{2}(I_{N})} =
O\bigg(\frac{1}{\sqrt{M}}\bigg),
\end{equation*}

\item \label{it:rNM0'-rN' L2 norm}
\begin{equation*}
\|{r_{N}^{M,0}}' -r_{N}'\|_{L^{2}(I_{N})} =
O\bigg(\frac{1}{M^{1/4}}\bigg).
\end{equation*}

\end{enumerate}
\end{lemma}

\begin{proof}

First, we notice that \eqref{it:rNM'-rN' L2 norm} (resp.
\eqref{it:rNM0'-rN' L2 norm}) follows from \eqref{it:rNM-rN L2 norm}
with \eqref{it:rNM''-rN'' L2 norm} (resp. \eqref{it:rNM0-rN L2 norm}
with \eqref{it:rNM0''-rN'' L2 norm}) by integration by parts and the
Cauchy-Schwartz inequality.

By the symmetry of all the functions involved, it is sufficient to
bound $\|\cdot\|_{L^{2}(I_{N}^{+})}$. To establish \eqref{it:rNM-rN
L2 norm}, we note that for $|x| \le M$, one has
\begin{equation}
\label{eq:SM=1+O(x/M)^2 org}
S_{M}(x) = 1+O\bigg(\frac{x}{M}\bigg)^2,
\end{equation}
and both $r_{N}$ and
$r_{N}^{M}$ are rapidly decaying for large $x$, since $$|r_{N}(x)|
\ll \frac{1}{x},$$ and $S_{M}$ is bounded, with constants
independent of $N$ or $M$. Thus, we have
\begin{equation*}
\begin{split}
&\|r_{N}^{M}-r_{N} \|^{2}_{L^{2}(I_{N})} = \int\limits_{0}^{\pi m}
(r_{N}^{M}(x)-r_{N}(x))^2 dx \\&= \int\limits_{0}^{\pi m}
(r_{N}(x)(1-S_{M}(x)))^2 dx \ll \frac{1}{M^4} \int\limits_{0}^{M}
r_{N}(x)^2 x^4 dx + \int\limits_{M}^{\pi m} r_{N}(x)^2 dx \\& \ll
\frac{1}{M^4} \int\limits_{0}^{M} (r_{N}(x)^2 x^2) \cdot x^2 dx +
\int\limits_{M}^{\pi m} \frac{dx}{x^2} \ll \frac{1}{M^4}
\int\limits_{0}^{M}x^2 dx + \frac{1}{M} \ll \frac{1}{M}.
\end{split}
\end{equation*}
It is easy to establish \eqref{it:rNM''-rN'' L2 norm} using a
similar approach.

To prove \eqref{it:rNM0-rN L2 norm}, we will use the Fourier series
representation \eqref{eq:rN Fourier repr} of $r_{N}$, and its
analogue \eqref{eq:rNM0 Fourier repr} for $r_{N}^{M,0}$ with
Parseval's identity. We then have by \eqref{eq:hatrNM0}
\begin{equation*}
\begin{split}
&\|r_{N}^{M,0}-r_{N} \|^{2} _{L^{2}(I_{N})} =
\|\hat{r}_{N}^{M,0}-\hat{r}_{N} \|_{l^{2}(\Z)} ^2 \\&=
2\sum\limits_{n=1}^{N} \hat{r}_{N}(n)\cdot
\big(\sqrt{\hat{r}_{N}(n)} - \sqrt{\hat{r}_{N}^{M}(n)}  \big)^{2}
\le 2\sum\limits_{n=1}^{N} \hat{r}_{N}(n) \cdot \big| \hat{r}_{N}(n)
- \hat{r}_{N}^{M}(n)\big|,
\end{split}
\end{equation*}
since for $a,b\ge 0$,
\begin{equation}
\label{eq:(a-b)^2 <= |a^2-b^2|} (a-b)^2\le |a^2-b^2|.
\end{equation}
Continuing, we use the Cauchy-Schwartz inequality to obtain
\begin{equation*}
\begin{split}
\|r_{N}^{M,0}-r_{N} \|^{2} _{L^{2}(I_{N})} &\ll \| \hat{r}_{N}
\|_{l^{2}(\Z)} \cdot \|\hat{r}_{N}-\hat{r}_{N}^{M} \|_{l^{2}(\Z)}
\\&= \| \hat{r}_{N} \|_{l^{2}(\Z)} \cdot \|r_{N}-r_{N}^{M} \|_{l^{2}(\Z)} \ll \frac{1}{\sqrt{M}},
\end{split}
\end{equation*}
by \eqref{it:rNM-rN L2 norm} of the present Lemma, and the obvious
estimate $\| \hat{r}_{N}\| \ll 1$. This proves part
\eqref{it:rNM0-rN L2 norm} of this Lemma.

It is now easy to establish part \eqref{it:rNM0''-rN'' L2 norm} of
the present Lemma, using $$\hat{f''}(n) = -\frac{n^2}{m^2}
\hat{f}(n).$$

\end{proof}

\begin{lemma}
\label{lem:r,rM,rM0 Lip} The functions $r_{N}(x)$, $r_{N}^{M}(x)$,
$r_{N}^{M,0}(x)$ and their first couple of derivatives are
Lipschitz, uniformly with $N$, i.e. satisfy
\begin{equation}
\label{eq:A-Lipschitz} |h(x)-h(y)| \le A|x-y|
\end{equation}
for some universal constant $A>0$.
\end{lemma}

\begin{proof}
The statement is clear for $r_{N}(x) =
\frac{1}{N}\sum\limits_{n=1}^{N} \cos(\frac{n}{m} x)$, as well as
$r_{N}^{M}(x) = r_{N}(x) S_{M}(x)$ (due to the fact that $S_{M}$ and
$S_{M}'$ are bounded).

It then remains to prove the result for $r_{N}^{M,0}$. From the
representation \eqref{eq:rNM0 Fourier repr} it is clear that it
would be sufficient to prove that the coefficients
$$\frac{1}{(2\pi m)^{1/4}} \cdot \sqrt{(\chi_{1\le |n| \le N} * \hat{S}_{M})(n)}$$
are uniformly bounded. We will bound the square
\begin{equation}
\label{eq:norm Fourier coeff rNM0} \frac{1}{(2\pi m)^{1/2}} \cdot
(\chi_{1\le |n| \le N} * \hat{S}_{M})(n).
\end{equation}
Using \eqref{eq:SM Fourier transf comp} with \eqref{eq:chi[-M,M]
Fourier transf}, we bound $\hat{S}_{M}(n)$ by
\begin{equation}
\label{eq:SM(n) est} \hat{S}_{M}(n) \ll
\begin{cases}
\frac{m^{15/2}}{M^{7}} \bigg(\frac{\sin\big(\frac{nM}{m}\big)}{n}
\bigg)^{8} , \: &n\ne 0 \\\frac{M}{\sqrt{m}},\: &n=0
\end{cases}
\end{equation}
so that the coefficients \eqref{eq:norm Fourier coeff rNM0} are
bounded by
\begin{equation*}
\begin{split}
&\ll \frac{1}{\sqrt{N}} \cdot \frac{m^{15/2}}{M^{7}}
\sum\limits_{\substack{k=n-N \\ k\ne 0}}^{n+N}
\bigg(\frac{\sin\big(\frac{kM}{m}\big)}{k} \bigg)^{8}
+\frac{1}{\sqrt{m}} \cdot \frac{M}{\sqrt{m}}
\\&\ll \frac{M}{N} \sum\limits_{1\le |k| \le \frac{N}{M}}
\bigg(\frac{\sin\big(\frac{kM}{m}\big)}{\frac{k M}{m}} \bigg)^{8} +
\frac{N^{7}}{M^{7}} \sum\limits_{|k| > \frac{N}{M}} \frac{1}{k^{8}}
+\frac{M}{N}
\\&\ll \frac{M}{N} \cdot \frac{N}{M} + \frac{N^{7}}{M^{7}} \cdot
\frac{1}{(N/M)^7} +1  \ll 1.
\end{split}
\end{equation*}
This proves that the squared coefficients \eqref{eq:norm Fourier
coeff rNM0} are uniformly bounded, and thus, that $r_{N}^{M,0}$
satisfy the Lipshitz condition \eqref{eq:A-Lipschitz} with some
universal constant $A>0$.

\end{proof}

\begin{lemma}
\label{lem:small L2 norm Lipschitz} Let $I=[a,b]$ be any interval
and $h:I\rightarrow\R$ a Lipschitz function satisfying
\eqref{eq:A-Lipschitz}. Then for every $x\in I$,
\begin{equation*}
|h(x)| \le 2A^{1/3}\| h\|_{L^{2}(I)} ^{2/3},
\end{equation*}
provided that
\begin{equation}
\label{eq:b-a>|h|/2A} b-a > \frac {\max\limits_{x\in
I}{\big|h(x)\big|}}{2A}.
\end{equation}

\end{lemma}

\begin{proof}
Let $x_{0}\in I$ and $$J:= I\cap
\bigg[x_{0}-\frac{|h(x_{0})|}{2A},x_{0}+\frac{|h(x_{0})|}{2A}\bigg].$$
Our assumption \eqref{eq:b-a>|h|/2A} implies that interval $J$ has
length $|J| \ge \frac{|h(x_{0})|}{2A}$, and moreover, on $J$ we have
\begin{equation*}
|h(x)| \ge \frac{|h(x_{0})|}{2}
\end{equation*}
by \eqref{eq:A-Lipschitz}. Thus we have
\begin{equation*}
\|h\|_{L^{2}(I)}^{2} \ge \int\limits_{J} h(x)^2 dx \ge |J| \cdot
\frac{h(x_{0})^2}{4} \ge \frac{|h(x_{0})|}{2A} \cdot
\frac{h(x_{0})^2}{4} = \frac{1}{8A} |h(x_{0})|^3,
\end{equation*}
which is equivalent to the statement of this Lemma.
\end{proof}

Lemmas \ref{lem:|rNM-rN|L2->0},  \ref{lem:r,rM,rM0 Lip} and
\ref{lem:small L2 norm Lipschitz} together imply
\begin{corollary}
\label{cor:rM->r unif}

For every $x\in I_{N}$, one has
\begin{equation*}
|r_{N}^{M}(x)-r_{N}(x)| = O\bigg(\frac{1}{M^{1/3}}\bigg),
\end{equation*}
\begin{equation*}
|r_{N}^{M,0}(x)-r_{N}(x)| = O\bigg(\frac{1}{M^{1/6}}\bigg),
\end{equation*}
\begin{equation*}
|{r_{N}^{M}}'(x)-r_{N}'(x)| = O\bigg(\frac{1}{M^{1/3}}\bigg),
\end{equation*}
\begin{equation*}
|{r_{N}^{M,0}}'(x)-r_{N}'(x)| = O\bigg(\frac{1}{M^{1/6}}\bigg),
\end{equation*}
\begin{equation*}
|{r_{N}^{M}}''(x)-r_{N}''(x)| = O\bigg(\frac{1}{M^{1/3}}\bigg),
\end{equation*}
\begin{equation*}
|{r_{N}^{M,0}}''(x)-r_{N}''(x)| = O\bigg(\frac{1}{M^{1/6}}\bigg),
\end{equation*}
uniformly w.r.t. $x$ and $N$.
\end{corollary}

\begin{lemma}
\label{eq:YNM->YN L2} For every $x\in I$, we have
\begin{equation*}
\E\big(Y_{N}^{M}(x)-Y_{N}(x)\big)^2 =
O\bigg(\frac{1}{\sqrt{M}}\bigg)
\end{equation*}
with the constant involved in the $O$-notation universal.
\end{lemma}

\begin{proof}
By the stationarity we may assume that $x=0$. We have by
\eqref{eq:spect form YN} and \eqref{eq:spect form YNM}
\begin{equation*}
\E\big(Y_{N}^{M}(0)-Y_{N}(0)\big)^2 = \frac{1}{\sqrt{2\pi
m}}\hat{r}_{N}^{M}(0)+ \frac{\sqrt{2}}{\sqrt{\pi
m}}\sum\limits_{n=1}^{\infty} \bigg(\sqrt{\hat{r}_{N}(n)} -
\sqrt{\hat{r}_{N}^{M}(n)} \bigg) ^{2}.
\end{equation*}
Since $\hat{r}_{N}$ is supported in $n\le N$, we have
\begin{equation}
\label{eq:E(YM-Y)^2 expl spl} \E\big(Y_{N}^{M}(x)-Y_{N}(x)\big)^2
\le \frac{\sqrt{2}}{\sqrt{\pi m}}\sum\limits_{n=0}^{2N}
\bigg(\sqrt{\hat{r}_{N}(n)} - \sqrt{\hat{r}_{N}^{M}(n)} \bigg) ^{2}
+ \frac{\sqrt{2}}{\sqrt{\pi m}} \sum\limits_{n>2N}
\hat{r}_{N}^{M}(n).
\end{equation}

We use \eqref{eq:(a-b)^2 <= |a^2-b^2|} again to bound the first
summation of \eqref{eq:E(YM-Y)^2 expl spl}, getting
\begin{equation*}
\begin{split}
\frac{\sqrt{2}}{\sqrt{\pi m}}\sum\limits_{n=0}^{2N} &\le
\frac{\sqrt{2}}{\sqrt{\pi m}} \sum\limits_{n=1}^{2N}
\big|\hat{r}_{N}(n) - \hat{r}_{N}^{M}(n) \big| \\&\ll
\frac{1}{\sqrt{N}} \cdot \sqrt{N} \bigg( \sum\limits_{n=1}^{2N}
\big(\hat{r}_{N}(n)-\hat{r}_{N}^{M}(n)\big)^{2} \bigg)^{1/2}
\\&\le \| \hat{r}_{N}-\hat{r}_{N}^{M} \|_{l^2(\Z)}  \ll \frac{1}{\sqrt{M}},
\end{split}
\end{equation*}
by the Cauchy-Schwartz inequality, Parseval's identity and Lemma
\ref{lem:|rNM-rN|L2->0}, part \eqref{it:rNM-rN L2 norm}.

To bound the second summation in \eqref{eq:E(YM-Y)^2 expl spl}, we
reuse the estimate \eqref{eq:SM(n) est} to obtain
\begin{equation*}
\hat{S}_{M}(n) \ll \frac{N^{15/2}}{M^{7}} \cdot \frac{1}{n^8} ,\:
n\ne 0
\end{equation*}
so that \eqref{eq:rM Fourier trans} implies
\begin{equation*}
\begin{split}
\hat{r}_{N}^{M}(n) &\ll \frac{1}{N} \sum\limits_{k=-N}^{N}
\hat{S}_{M}(n+k) \ll \frac{1}{N} \sum\limits_{k=-N}^{N}
\frac{N^{15/2}}{M^{7}} \cdot \frac{1}{(n+k)^{8}} \\&\ll
\frac{N^{13/2}}{M^{7}} \cdot N\frac{1}{(n/2)^{8}} \ll
\frac{N^{15/2}}{M^{7}}\cdot \frac{1}{n^{8}},
\end{split}
\end{equation*}
and thus the second summation in \eqref{eq:E(YM-Y)^2 expl spl} is
bounded by
\begin{equation*}
\frac{1}{\sqrt{N}}\sum\limits_{n>2N} \frac{N^{15/2}}{M^{7}}\cdot
\frac{1}{n^{8}} \ll \frac{1}{M^7}.
\end{equation*}

This concludes the proof of this lemma.

\end{proof}

\begin{lemma}[Cuzick ~\cite{CZ}, lemma 4]
\label{lem:Cuzick Cor abs} Let $V_{1}$ and $V_{2}$ be a mean zero
Gaussian pair of random variables and let
\begin{equation*}
\rho = Cor(V_{1},V_{2}) := \frac{\E V_{1}V_{2}-\E V_{1}\E
V_{2}}{\sqrt{\Var V_{1} \cdot \Var V_{2}}}.
\end{equation*}
Then
\begin{equation*}
0 \le Cor(|V_{1}|,|V_{2}|) \le \rho^2.
\end{equation*}
\end{lemma}

\begin{remark}
Lemma \ref{lem:Cuzick Cor abs} may be also obtained computing explicitly
both sides of the inequality using the integral \eqref{eq:int |z1||z2| Gauss} due to Bleher and Di
~\cite{BD}.

\end{remark}

\subsection{Proof of Proposition \ref{prop:E(Z-Zmol)^2=o(N)}}

\begin{proof}[Proof of Proposition \ref{prop:E(Z-Zmol)^2=o(N)}]

Recall that $Z_{N}$ (resp. $Z_{N}^{M}$) is the number of the zeros
of $Y_{N}$ (resp. $Y_{N}^{M}$) on $I_{N}=[-\pi m, \pi m]$. The
process $Y_{N}^{M}$ is given in its spectral form \eqref{eq:spect
form YNM} with the RHS absolutely convergent, uniformly w.r.t. $x\in
I_{N}$. The rapid decay of $\hat{r}_{N}^{M}$ implies that
$Y_{N}^{M}$ is almost surely continuously differentiable, and we may
differentiate \eqref{eq:spect form YNM} term by term.

As stated before, for notational convenience, rather than showing
the original statement of the Proposition, we are going to prove
\eqref{eq:E(Z+-ZM+)^2}. We want to bound $E_{i}$ defined by
\eqref{eq:E1 def}-\eqref{eq:E3 def} given the large integral
parameters $S$ and $R$.

\subsubsection{Bounding $E_{1}$}
For every $x\in I_{N}$, let $\chi_{N}^{x}$ (resp. $\chi_{N,M}^{x}$)
be the indicator of the event $\{ Y_{N}(0)Y_{N}(x) < 0 \}$ (resp.
$\{ Y_{N}^{M}(0)Y_{N}^{M}(x) < 0 \}$). Intuitively, for $S$ large
(i.e. $I_{N,1}$ is short) one expects at most one zero of either
$Y_{N}$ or $Y_{N}^{M}$ on $I_{N,1}$. Thus the number of zeros of
$Y_{N}$ (resp. $Y_{N}^{M}$) on $I_{N,1}= [0,\tau]$ with
$\tau:=\frac{\pi}{2S} $ should, with high probability, equal
$\chi=\chi_{N}^{\tau}$ (resp. $\chi^{M}=\chi_{N,M}^{\tau}$).

Using the Cauchy-Schwartz inequality, we have for every
$k_{1},k_{2}$
\begin{equation*}
\begin{split}
&Cov(Z_{N,k_1}-Z_{N,k_1}^{M}, Z_{N,k_2}-Z_{N,k_2}^{M}) \\&\le
\sqrt{\Var(Z_{N,k_1}-Z_{N,k_1}^{M})} \cdot
\sqrt{\Var(Z_{N,k_2}-Z_{N,k_2}^{M})} = \Var(Z_{N,1}-Z_{N,1}^{M})
\end{split}
\end{equation*}
by the stationarity. Therefore,
\begin{equation}
\label{eq:E1 << trngl}
\begin{split}
E_{1} &\ll SN \Var(Z_{N,1}-Z_{N,1}^{M}) \le
SN\E(Z_{N,1}-Z_{N,1}^{M})^2
\\&\ll SN\bigg(\E (Z_{N,1} - \chi)^2+\E(\chi - \chi^{M})^2+\E(\chi^{M}-Z_{N,1}^{M})^2\bigg).
\end{split}
\end{equation}

We recognize the second summand of \eqref{eq:E1 << trngl} as the
probability $Pr(\chi\ne \chi^{M})$. We may bound it as
\begin{equation}
\label{eq:chi!=chiM<=sgn0+sgnx0}
\begin{split}
\E(\chi - \chi^{M})^2 &=Pr(\chi\ne \chi^{M}) \le
Pr\big(\mbox{sgn}(Y_{N}(0)) \ne
\mbox{sgn}(Y_{N}^{M}(0))\big)\\&+Pr\big(\mbox{sgn}(Y_{N}(\tau)) \ne
\mbox{sgn}(Y_{N}^{M}(\tau))\big).
\end{split}
\end{equation}
We bound the first summand of the RHS of
\eqref{eq:chi!=chiM<=sgn0+sgnx0}, and similarly the second one. For
every $\epsilon>0$, we have
\begin{equation}
\label{eq:sgn0<=eps} Pr\big(\mbox{sgn}(Y_{N}(0)) \ne \mbox{sgn}(Y_{N}^{M}(0))\big)
\le Pr(|Y_{N}(0)| < \epsilon)+Pr(|Y_{N}(0)-Y_{N}^{M}(0)| >
\epsilon).
\end{equation}
The first summand of \eqref{eq:sgn0<=eps} is bounded by
\begin{equation*}
Pr(|Y_{N}(0)| < \epsilon) =O(\epsilon),
\end{equation*}
since $Y_{N}(0)$ is $(0,1)$-Gaussian, and the second one is
\begin{equation*}
Pr(|Y_{N}(0)-Y_{N}^{M}(0)| > \epsilon) \ll
\frac{1}{\sqrt{M}\epsilon^2},
\end{equation*}
by Lemma \ref{eq:YNM->YN L2} and Chebyshev's inequality.

Hence, we obtain the bound
\begin{equation*}
Pr\big(\mbox{sgn}(Y_{N}(0)) \ne \mbox{sgn}(Y_{N}^{M}(0))\big) = O\bigg(\epsilon +
\frac{1}{\sqrt{M}\epsilon^2}\bigg),
\end{equation*}
and, similarly,
\begin{equation*}
Pr\big(mbox{sgn}(Y_{N}(\tau)) \ne
\mbox{sgn}(Y_{N}^{M}(\tau))\big)=O\bigg(\epsilon +
\frac{1}{\sqrt{M}\epsilon^2}\bigg).
\end{equation*}
Plugging the last couple of estimates into
\eqref{eq:chi!=chiM<=sgn0+sgnx0} yields that for every $\epsilon >
0$
\begin{equation}
\label{eq:E(chi-chiM)2 est} \E(\chi - \chi^{M})^2 = O\bigg(\epsilon
+ \frac{1}{\sqrt{M}\epsilon^2}\bigg).
\end{equation}
The RHS of \eqref{eq:E(chi-chiM)2 est} can be made arbitrarily
small.

Now we treat the third summand of \eqref{eq:E1 << trngl}, and
similarly, but easier, the first one. We have
\begin{equation}
\begin{split} \label{eq:E(Z-chi)^2 rel Z} \E(Z_{N,1}^{M}-\chi^{M})^2 &=
\sum\limits_{k=1}^{\infty}k^2 Pr(Z_{N,1}^{M}-\chi^{M} = k) \\&\le
\sum\limits_{k=2}^{\infty}2(k^2-k) Pr(Z_{N,1}^{M} = k+\chi^{M}) \le
2\E \big[(Z_{N,1}^{M})^2-Z_{N,1}^{M} \big],
\end{split}
\end{equation}
and and using the same approach as in \eqref{eq:bnd par sum var} in
addition to some easy manipulations yields
\begin{equation*}
\E \big[(Z_{N,1}^{M})^2-Z_{N,1}^{M} \big] \ll \int\limits_{0}^{\tau}
(\tau-x) \cdot \tilde{K}(x) dx ,
\end{equation*}
recalling the notation $\tau:=\frac{\pi}{2S}$, where
$$\tilde{K}(x) = \frac{{\lambda_{2,N}^{M}}'(1-{r}^2)-{r'}^2}{{(1-{r}^2)}^{3/2}}\big(\sqrt{1-\rho^2}+\rho\arcsin{\rho}\big) $$
with notations as in \eqref{eq:bnd par sum var int man}. We saw
already that $\tilde{K}(x)$ is bounded, uniformly w.r.t. $N$, so
that
\begin{equation*}
\E \big[(Z_{N,1}^{M})^2-Z_{N,1}^{M} \big] =O(\tau^2) =
O\bigg(\frac{1}{S^2} \bigg).
\end{equation*}
Plugging the last estimate into \eqref{eq:E(Z-chi)^2 rel Z}, we
obtain the bound
\begin{equation}
\label{eq:E(ZN1M-chiM)^2 bnd} \E(Z_{N,1}^{M}-\chi^{M})^2 =
O\bigg(\frac{1}{S^2}\bigg)
\end{equation}
and similarly,
\begin{equation}
\label{eq:E(ZN1-chi)^2 bnd} \E(Z_{N,1}-\chi)^2 =
O\bigg(\frac{1}{S^2}\bigg)
\end{equation}
as well.

Collecting the bounds for various summands of \eqref{eq:E1 << trngl}
we encountered i.e. \eqref{eq:E(chi-chiM)2 est},
\eqref{eq:E(ZN1M-chiM)^2 bnd} and \eqref{eq:E(ZN1-chi)^2 bnd}, we
obtain the bound
\begin{equation*}
|E_{1}| \ll
NS\bigg(\epsilon+\frac{1}{\sqrt{M}\epsilon^2}+\frac{1}{S^2}\bigg),
\end{equation*}
or, equivalently,
\begin{equation}
\label{eq:E1/N bnd} \frac{|E_{1}|}{N} \ll \epsilon
S+\frac{S}{\sqrt{M}\epsilon^2}+\frac{1}{S},
\end{equation}
which could be made arbitrarily small.

\subsubsection{Bounding $E_{2}$}
We write $E_{2}$ as
\begin{equation}
\label{eq:Eprodsum=sumEprod}
\begin{split}
E_{2}&=\sum Cov\big(Z_{N,k_{1}}-Z_{N,k_{1}}^{M},
Z_{N,k_{2}}-Z_{N,k_{2}}^{M}\big) = \sum \E Z_{N,k_{1}}\cdot
\left(Z_{N,k_{2}}-Z_{N,k_{2}}^{M}\right) \\&- \sum\E Z_{N,k_{1}}^{M}\cdot
\left(Z_{N,k_{2}}-Z_{N,k_{2}}^{M}\right)-\sum \E \left[Z_{N,k_{1}}-Z_{N,k_{1}}^{M}\right]
\E\left[Z_{N,k_{2}}-Z_{N,k_{2}}^{M}\right] \\&=: E_{2,1}-E_{2,2}-E_{2,t},
\end{split}
\end{equation}
and bound each of the summands of \eqref{eq:Eprodsum=sumEprod}
separately. In fact, we will only bound the contribution of the
summands $E_{2,t}$ and of the (slightly more difficult of the
remaining two) $E_{2,2}$, bounding $E_{2,1}$ in a similar manner.
Note that the number of summands in each of the summations
in \eqref{eq:Eprodsum=sumEprod}, which equals the number of pairs
$(k_{1},k_{2})$ with $2 \le |k_{1}-k_{2}| \le RS$, is
of order $K\cdot SR = NS\cdot SR = N S^2 R$.

We reuse the notation $\tau:= \frac{\pi}{2 S}$.
Note that differentiating \eqref{eq:rNM def} yields
\begin{equation*}
{r_{N}^{M}}''(0)=r_{N}''(0)+O\left(\frac{1}{M^2}\right).
\end{equation*}
using \eqref{eq:SM=1+O(x/M)^2 org} near the origin.
One then has
\begin{equation*}
\begin{split}
\E [Z_{N,k_{1}}-Z_{N,k_{1}}^{M}] \E[Z_{N,k_{2}}-Z_{N,k_{2}}^{M}] &=
\big(\E [Z_{N,1}-Z_{N,1}^{M}]\big)^2 \\&\ll \big(\tau
(r_{N}''(0)-{r_{N}^{M}}''(0))\big)^2 \ll \frac{1}{S^2 M^4},
\end{split}
\end{equation*}
by the stationarity, formula \eqref{eq:exp num zer} and
its analogue for $Y_{N}^{M}$. Therefore
\begin{equation}
\label{eq:E2t bnd} E_{2,t} \ll N\frac{R}{M^4}.
\end{equation}

Note that for $|k_{1}-k_{2}|\ge 2$, the intervals $I_{N,k_{1}}$ and
$I_{N,k_{2}}$ are disjoint. Using a similar approach to ~\cite{CL},
we find that
\begin{equation*}
\begin{split}
&\E Z_{N,k_{1}}^{M}\cdot (Z_{N,k_{2}}-Z_{N,k_{2}}^{M}) =
\int\limits_{(k_{2}-k_{1}-1)\tau}^{(k_{2}-k_{1}+1)\tau}
(\tau-|x-(k_{2}-k_{1})\tau |) \times \\&\times
\bigg[\iint\limits_{\R^{2}} |v_{1}|\cdot |v_{2}|
\big(\tilde{\phi}_{N,M,0}^{x}(v_{1},v_{2})-\tilde{\phi}_{N,M}^{x}(v_{1},v_{2})\big)
dv_{1}dv_{2}\bigg] dx,
\end{split}
\end{equation*}
using the notations \eqref{eq:tildphiNM,NM0 def}.
We then bound $E_{2,2}$ as
\begin{equation} \label{eq:E22 bnd int}
\begin{split}  E_{2,2} &\le NS^2 R \cdot \max\limits_{|k_{1}-k_{2}|\ge 2} \big\{ \E
Z_{N,k_{1}}^{M}\cdot (Z_{N,k_{2}}-Z_{N,k_{2}}^{M}) \big\}
\\& \ll N R \max\limits_{\tau \le x \le \pi m}
\bigg\{\iint\limits_{\R^2} |v_{1}| |v_{2}| \cdot \big|
\tilde{\phi}_{N,M,0}^{x}(v_{1},v_{2})-\tilde{\phi}_{N,M}^{x}(v_{1},v_{2})
\big| dv_{1}dv_{2} \bigg\},
\end{split}
\end{equation}
where we used the obvious inequality
$$\max\limits_{2\le |k_{1}-k_{2}| \le RS} \le  \max\limits_{ |k_{1}-k_{2}| \ge 2}  .$$

To bound the last integral, we exploit the fact that on any compact
subset of $\R^{2}$ we have
$$|\tilde{\phi}_{N,M,0}^{x}(v_{1},v_{2})-\tilde{\phi}_{N,M}^{x}(v_{1},v_{2})|\rightarrow
0$$ as $N\rightarrow\infty$, uniformly w.r.t. $x>\tau$, whereas
outside both $\tilde{\phi}_{N,M,0}^{x}$ and $\tilde{\phi}_{N,M}^{x}$
are rapidly decaying. More precisely, let $T>0$ be a large
parameter. We write
\begin{equation}
\label{eq:intR2 J1+J2} \iint\limits_{\R^2} = \iint\limits_{[-T,T]^2}
+ \iint\limits_{\max\{|v_{i}| \} \ge T} =: J_{1}+J_{2}.
\end{equation}

While bounding $J_{2}$, we may assume with no loss of generality,
that $$|v_{1}| \ge T$$ on the domain of the integration. Let
\begin{equation*}
J_{2,1} := \iint\limits_{|v_{1}|\ge T} |v_{1}| |v_{2}| \cdot \big|
\tilde{\phi}_{N,M,0}^{x}(v_{1},v_{2}) \big| dv_{1}dv_{2}
\end{equation*}
and
\begin{equation*}
J_{2,2} := \iint\limits_{|v_{1}|\ge T} |v_{1}| |v_{2}| \cdot \big|
\tilde{\phi}_{N,M}^{x}(v_{1},v_{2}) \big| dv_{1}dv_{2},
\end{equation*}
so that
\begin{equation}
\label{eq:J2<=J21+J22} J_{2} \le J_{2,1}+J_{2,2}.
\end{equation}

Upon using the Cauchy-Schwartz inequality, we obtain
\begin{equation}
\label{eq:J22 bnd CS}
\begin{split}
J_{2,2} &\le  \int\limits_{-\infty}^{\infty} dv_{2}
\int\limits_{|v_{1}|\ge T} |v_{1}| |v_{2}| \cdot
\tilde{\phi}_{N,M}^{x}(v_{1},v_{2}) dv_{1}  \\&\ll
\bigg(\int\limits_{-\infty}^{\infty} v_{2}^2 dv_{2}
\int\limits_{T}^{\infty}\tilde{\phi}_{N,M}^{x}(v_{1},v_{2}) dv_{1}
\bigg)^{1/2} \cdot \bigg( \int\limits_{-\infty}^{\infty} dv_{2}
\int\limits_{T}^{\infty} v_{1}^2\tilde{\phi}_{N,M}^{x}(v_{1},v_{2})
dv_{1} \bigg)^{1/2}
\\&\le \bigg(\frac{\E\big[{Y_{N}^{M}}'(x)^2\big| Y_{N}^{M}(0)=Y_{N}^{M}(x)=0 \big]}{2\pi\sqrt{1-r_{N}^{M}(x)^2}} \bigg)^{1/2} \cdot
\bigg( \int\limits_{-\infty}^{\infty} dv_{2}
\int\limits_{T}^{\infty} v_{1}^2\tilde{\phi}_{N,M}^{x}(v_{1},v_{2})
dv_{1} \bigg)^{1/2},
\end{split}
\end{equation}
by \eqref{eq:phiNM=psiNM/det}.

Computing explicitly, we have
\begin{equation}
\label{eq:cond E bnd J22} \E\big[{Y_{N}^{M}}'(x)^2\big|
Y_{N}^{M}(0)=Y_{N}^{M}(x)=0 \big] = {\lambda_{2,N}^{M}}' -
\frac{{r_{N}^{M}}'(x)^2}{1-r_{N}^{M}(x)^2} = O(1),
\end{equation}
where ${\lambda_{2,N}^{M}}' := -{r_{N}^{M}}''(0)$, and, changing the
order of integration,
\begin{equation*}
\int\limits_{-\infty}^{\infty} dv_{2} \int\limits_{T}^{\infty}
v_{1}^2\tilde{\phi}_{N,M}^{x}(v_{1},v_{2}) dv_{1} =
\int\limits_{T}^{\infty} v^2 \exp\bigg(-\frac{1}{2}
\frac{v^2}{\sigma^2}\bigg)\frac{dv}{2\pi
\sigma\sqrt{1-r_{N}^{M}(x)^2}},
\end{equation*}
where
\begin{equation*}
\sigma^2 := \E\big[{Y_{N}^{M}}'(0)^2 \big|
Y_{N}^{M}(0)=Y_{N}^{M}(x)=0\big]={\lambda_{2,N}^{M}}' -
\frac{{r_{N}^{M}}'(x)^2}{1-r_{N}^{M}(x)^2}
\end{equation*}
as well. Continuing, we bound the integral by
\begin{equation}
\begin{split}
\label{eq:int >=T bnd J22} \int\limits_{-\infty}^{\infty} dv_{2}
\int\limits_{T}^{\infty} &\ll
\frac{\sigma^2}{\sqrt{1-r_{N}^{M}(x)^2}}\int\limits_{\frac{T}{\sigma}}^{\infty}
v'^2 \exp\bigg(-\frac{1}{2}v'^2\bigg) dv' \\&\ll
\frac{\sigma^2}{T^2} \cdot
\frac{\sigma^{2}}{\sqrt{1-r_{N}^{M}(x)^2}} \ll \frac{\sigma^{4}}{T^2
\sqrt{1-r_{N}^{M}(x)^2}}
\end{split}
\end{equation}
(say), by the rapid decay of the exponential.

Plugging \eqref{eq:cond E bnd J22} and \eqref{eq:int >=T bnd J22}
into \eqref{eq:J22 bnd CS}, and using the crude estimate
$$1-r_{N}^{M}(x) \gg \tau^2$$ for $\tau \le x \le \pi m$, we obtain
the estimate
\begin{equation}
\label{eq:J22 bnd} J_{2,2} \ll \frac{\bigg({\lambda_{2,N}^{M}}' -
\frac{{r_{N}^{M}}'(x)^2}{1-r_{N}^{M}(x)^2}\bigg)^{3/2}}{\sqrt{1-{r_{N}^{M}}(x)^2}}
\cdot \frac{1}{T} \ll \frac{S}{T}.
\end{equation}

Repeating all of the above for $J_{2,1}$, we obtain
\begin{equation}
\label{eq:J21 bnd} J_{2,1} \ll \frac{\bigg({\lambda_{2,N}}' -
\frac{{r_{N}^{M}}'(x)^2}{1-r_{N}^{M,0}(x)^2}\bigg)\cdot
\bigg({\lambda_{2,N}^{M}}' -
\frac{{r_{N}^{M,0}}'(x)^2}{1-r_{N}^{M,0}(x)^2}\bigg)
^{1/2}}{\sqrt{1-{r_{N}^{M,0}}(x)^2}} \cdot \frac{1}{T} \ll
\frac{S}{T},
\end{equation}
using the same estimate $$1-{r_{N}^{M,0}}(x) \gg \tau^2,$$ which is
easy to obtain using \eqref{eq:rNM0 Fourier repr} and Lemmas
\ref{lem:|rNM-rN|L2->0} and \ref{lem:r,rM,rM0 Lip}.

Plugging the inequality \eqref{eq:J21 bnd} together with
\eqref{eq:J22 bnd} into \eqref{eq:J2<=J21+J22}, we obtain
\begin{equation}
\label{eq:J2 bnd} J_{2} = O\bigg(\frac{S}{T}\bigg).
\end{equation}

Now we are going to bound $J_{1}$. Recall the definition
\eqref{eq:tildphiNM,NM0 def} of $\tilde{\phi}^{x}_{N,M}$ and
$\tilde{\phi}^{x}_{N,M,0}$ with \eqref{eq:phi NM def} and
\eqref{eq:phi NM0 def}, and the covariance matrices \eqref{eq:Sigma
NM def}, \eqref{eq:Sigma NM0 def}.

Corollary \ref{cor:rM->r unif} implies
$$|\Sigma_{N}^{M} - \Sigma_{N}^{M,0}| = O\bigg(\frac{1}{M^{1/6}}\bigg) $$
(here and anywhere else the inequality $M \le y$ where $M$ is a
matrix and $y$ is a number means that all the entries of $M$ are
$\le$ than $y$). Expanding the determinants $\det{\Sigma_{N}^{M}}$
and $\det{\Sigma_{N}^{M,0}}$ into Taylor polynomial around the
origin shows that they are bounded away from zero in the sense that
for $\tau < x < \pi m$,
\begin{equation*}
\det{\Sigma_{N}^{M}},\, \det{\Sigma_{N}^{M,0}} \gg \tau^A \gg
\frac{1}{S^A}
\end{equation*}
for some constant $A>0$. 

\begin{remark}
An explicit
computation shows that $\det\Sigma_{N}^{M}(x) \gg x^8$ and also
$\det\Sigma_{N}^{M,0}(x) \gg x^8$, with universal constants.
\end{remark}

Thus also $$|(\Sigma_{N}^{M})^{-1} - (\Sigma_{N}^{M,0})^{-1}| =
O\bigg(\frac{S^A}{M^{1/6}}\bigg)$$ and
\begin{equation*}
\bigg|\frac{1}{\sqrt{\det{\Sigma_{N}^{M}}}} -
\frac{1}{\sqrt{\det{\Sigma_{N}^{M,0}}}} \bigg| \ll
O\bigg(\frac{S^{2A}}{M^{1/6}}\bigg)
\end{equation*}

Substituting the estimates above into \eqref{eq:phi NM def} and
\eqref{eq:phi NM0 def}, and using \eqref{eq:tildphiNM,NM0 def}, we
obtain
\begin{equation*}
\big|\tilde{\phi}_{N,M,0}^{x}(v_{1},v_{2})-\phi_{N,M}^{x}(v_{1},v_{2})\big|
\ll \frac{S^{2A}}{M^{1/6}}+S^{A/2}\cdot \frac{T^2 S^{A}}{M^{1/6}}
\ll \frac{S^{2A} T^2}{M^{1/6}},
\end{equation*}
uniformly for $\tau \le x \le \pi m$ and $|v_{i}| \le T$, where we
used the trivial estimate $|e^{x}-e^{y}|\le |x-y|$ for $x,y < 0$.

Integrating the last estimate for $|v_{i}| \le T$ and substituting
into the definition of $J_{1}$, we finally obtain
\begin{equation}
\label{eq:J1 bnd} J_{1} = O\bigg( \frac{S^{2A} T^6}{M^{1/6}} \bigg).
\end{equation}

Upon combining \eqref{eq:J1 bnd} and \eqref{eq:J2 bnd}, and
recalling \eqref{eq:E22 bnd int} with \eqref{eq:intR2 J1+J2}, we
finally obtain a bound for $E_{2,2}$
\begin{equation}
\label{eq:E22 bnd} E_{2,2} = NR \cdot \bigg( O\bigg( \frac{S^{2A}
T^6}{M^{1/6}} \bigg) + O\bigg(\frac{S}{T}\bigg) \bigg),
\end{equation}
and repeating the same computation for $E_{2,1}$, we may find that
the same bound is applicable for $E_{2,1}$
\begin{equation}
\label{eq:E21 bnd} E_{2,1} = NR \cdot \bigg( O\bigg( \frac{S^{2A}
T^6}{M^{1/6}} \bigg) + O\bigg(\frac{S}{T}\bigg) \bigg).
\end{equation}

Using \eqref{eq:E21 bnd} together with \eqref{eq:E22 bnd} and
\eqref{eq:E2t bnd}, and noting \eqref{eq:Eprodsum=sumEprod}, we
finally obtain a bound for $E_{2}$
\begin{equation*}
E_{2} = NR \cdot \bigg( O\bigg( \frac{S^{2A} T^6}{M^{1/6}} \bigg) +
O\bigg(\frac{S}{T}\bigg) \bigg),
\end{equation*}
so that
\begin{equation}
\label{eq:E2/N bnd} \frac{E_{2}}{N} = O\bigg( \frac{R S^{2A}
T^6}{M^{1/6}} \bigg) + O\bigg(\frac{RS}{T}\bigg),
\end{equation}
which could be made arbitrarily small.

\subsubsection{Bounding $E_{3}$}

By the symmetry $$Cov (Z_{N,k_{1}}, Z_{N,k_{2}}^{M}) = Cov(
Z_{N,k_{2}}, Z_{N,k_{1}}^{M})$$ we may rewrite $E_{3}$ as
\begin{equation}
\label{eq:E3=E31-2E32+E33}
\begin{split}
E_{3} &= \sum Cov( Z_{N,k_{1}}, Z_{N,k_{2}}) - 2 Cov( Z_{N,k_{1}},
Z_{N,k_{2}}^{M}) + Cov( Z_{N,k_{1}}^{M}, Z_{N,k_{2}}^{M}) \\&=:
E_{3,1}-2E_{3,2}+E_{3,3}.
\end{split}
\end{equation}

First we treat the ``mixed" term $E_{3,2}$, providing a similar
treatment for the other terms. Assume with no loss of generality,
that $k_{2} > k_{1}$. Here we employ the random vector
$(V_{1},V_{2})$ defined in section \ref{sec:prob prelim cov}. Using
the theory developed in ~\cite{CL}, modified to treat the covariance
(see also remark \ref{rem:VarZ=intE|X'(0)X'(t)||0}), we may write
\begin{equation*}
\begin{split}
Cov( Z_{N,k_{1}}, Z_{N,k_{2}}^{M}) &=
\frac{1}{2\pi}\int\limits_{(k_{2}-k_{1}-1)\tau}^{(k_{2}-k_{1}+1)\tau}
\bigg[ \big(\tau-|x-(k_{2}-k_{1})\tau |\big) \times \\&
\bigg(\frac{\E\big[|V_{1}(x)V_{2}(x)|\big]}{\sqrt{1-r_{N}^{M,0}(x)^2}}-
\E |Y_{N}'(0)| \E|{Y_{N}^{M}}'(x)| \bigg) \bigg] dx.
\end{split}
\end{equation*}
where, as usual, we denote $\tau := \frac{\pi}{2S}$. Summing that up
for $|k_{2}-k_{1}| \ge SR$, and using the stationarity, we obtain
the bound
\begin{equation*}
E_{3,2} \ll N \int\limits_{\frac{\pi R}{2}}^{\pi m }
\bigg[\frac{\E\big[|V_{1}(x)V_{2}(x)|\big]}{\sqrt{1-r_{N}^{M,0}(x)^2}}-
\E |Y_{N}'(0)| \E|{Y_{N}^{M}}'(x)| \bigg] dx,
\end{equation*}
so that
\begin{equation}
\label{eq:E32 bnd int} \frac{E_{3,2}}{N} \ll \int\limits_{\frac{\pi
R}{2}}^{\pi m }
\bigg[\frac{\E\big[|V_{1}(x)V_{2}(x)|\big]}{\sqrt{1-r_{N}^{M,0}(x)^2}}-
\E |Y_{N}'(0)| \E|{Y_{N}^{M}}'(x)| \bigg] dx.
\end{equation}

To bound the integral on the RHS of \eqref{eq:E32 bnd int}, we use
the triangle inequality to write
\begin{equation}
\begin{split}
\label{eq:E32 bnd sumint} \frac{E_{3,2}}{N} &\ll
\int\limits_{\frac{\pi R}{2}}^{\pi m } \frac{\big|
Cov(|V_{1}|,|V_{2}|)\big| }{\sqrt{1-r_{N}^{M,0}(x)^2}}dx \\&+
\int\limits_{\frac{\pi R}{2}}^{\pi m } \bigg|\frac{\E |V_{1}| \cdot
\E |V_{2}|}{\sqrt{1-r_{N}^{M,0}(x)^2}}- \E |Y_{N}'(0)|
\E|{Y_{N}^{M}}'(x)| \bigg| dx =: J_{3,1}+J_{3,2}.
\end{split}
\end{equation}
For $\frac{\pi R}{2} < x< \pi m$, $R$ sufficiently large,
$r_{N}^{M,0}$ is bounded away from $1$ (see Corollary \ref{cor:rM->r
unif}), and therefore, while bounding $J_{3,1}$, we may disregard
the denominator of the first integrand in \eqref{eq:E32 bnd sumint}.
Note that if $V$ is a mean zero Gaussian random variable, then
\begin{equation}
\label{eq:E|V| V mean 0 Gaus} \E(|V|) = \sqrt{\frac{2}{\pi}} \cdot
\sqrt{\Var(V)}.
\end{equation}
and
\begin{equation}
\label{eq:Var|V| R mean 0 Gaus} \Var(|V|) =
\bigg(1-\frac{2}{\pi}\bigg) \Var(V).
\end{equation}
Note also that for $\frac{\pi R}{2} < x< \pi m$ the variances
$\Var(V_{1}(x))$ and $\Var(V_{2}(x))$, given by the diagonal entries
of \eqref{eq:OmegaNM0 cond def}, are bounded away from $0$. This
follows from the decay of $r_{N}^{M,0}(x)$ and ${r_{N}^{M,0}}'(x)$
for large values of $x$, due to Corollary \ref{cor:rM->r unif}.
Thus, an application of Lemma \ref{lem:Cuzick Cor abs} yields
\begin{equation*}
0 \le Cov(|V_{1}|,|V_{2}|) \le
\frac{(1-\frac{2}{\pi})Cov(V_{1},V_{2})^2}{\sqrt{\Var(V_{1})\cdot
\Var(V_{2})}} \ll Cov(V_{1},V_{2})^2.
\end{equation*}

All in all, we obtain the estimate
\begin{equation*}
J_{3,1} \ll \int\limits_{\frac{\pi V}{2}}^{\pi m }
Cov(V_{1},V_{2})^2 dx,
\end{equation*}
which (this time, using the off-diagonal elements of
\eqref{eq:OmegaNM0 cond def}) is
\begin{equation}
\label{eq:J31 bnd}
\begin{split}
J_{3,1} &\ll \int\limits_{\frac{\pi R}{2}}^{\pi m } \bigg[
{-r_{N}^{M,0}}''(x) -\frac{r_{N}^{M,0}(x)\cdot {r_{N}^{M,0}}'(x)^2
}{1-r_{N}^{M,0}(x)^2} \bigg]^2 dx \ll \int\limits_{\frac{\pi
R}{2}}^{\pi m }\big({r_{N}^{M,0}}''(x)^2 + {r_{N}^{M,0}(x)}^2\big)
dx \\&\ll \|r_{N}^{M,0}- r_{N}\|
_{L^{2}(I_{N})}^{2}+\|{r_{N}^{M,0}}''- r_{N}''\|
_{L^{2}(I_{N})}^{2}+ \int\limits_{\frac{\pi R}{2}}^{\pi m }
\big(r_{N}''(x)^2 + r_{N}(x)^2\big)dx \\&\ll \frac{1}{\sqrt{M}} +
\frac{1}{R},
\end{split}
\end{equation}
by the triangle inequality, Lemma \ref{lem:|rNM-rN|L2->0}, and the
decay
$$r_{N}(x),r_{N}''(x) \ll \frac{1}{x}.$$

To bound $J_{3,2}$, we note that for $\frac{\pi R}{2} < x < \pi m$,
we may expand
\begin{equation*}
\frac{1}{\sqrt{1-{r_{N}^{M,0}}(x)^2}} =
1+O\big({r_{N}^{M,0}}(x)^2\big),
\end{equation*}
with the constant involved in the $'O'$-notation being uniform,
since $r_{N}^{M,0}$ is bounded away from $1$ (by Corollary
\ref{cor:rM->r unif}, say). Thus we may use the triangle inequality
to write
\begin{equation}
\label{eq:J32 bnd trngl}
\begin{split}
J_{3,2} &\ll \int\limits_{\frac{\pi R}{2}}^{\pi m } \bigg|\E |V_{1}|
\cdot \E |V_{2}|- \E |Y_{N}'(0)| \E|{Y_{N}^{M}}'(x)| \bigg| dx +
\int\limits_{\frac{\pi R}{2}}^{\pi m } r_{N}^{M,0}(x)^2 \cdot \E
|V_{1}| \E |V_{2}| dx \\&\le \int\limits_{\frac{\pi R}{2}}^{\pi m }
\E |V_{1}| \big(\big| \E |V_{2}| - \E {|Y_{N}^{M}}'(x)| \big|\big)dx
+ \int\limits_{\frac{\pi R}{2}}^{\pi m } \E {|Y_{N}^{M}}'(x)|
\big(\big|\E |V_{1}| - \E {|Y_{N}}'(0)| \big|\big)dx
\\&+ \int\limits_{\frac{\pi R}{2}}^{\pi m }
r_{N}^{M,0}(x)^2 \cdot \E |V_{1}| \E |V_{2}| dx.
\end{split}
\end{equation}

Now, \eqref{eq:E|V| V mean 0 Gaus} allows us to compute $\E|V_{1}|$,
$\E|V_{2}|$, $\E |Y_{N}'(0)|$ and $\E |{Y_{N}^{M}}'(x)|$;
\eqref{eq:OmegaNM0 cond def} implies that all of the expectations
above are uniformly bounded for $\frac{\pi R}{2} < x < \pi m $. Thus
the third term of \eqref{eq:J32 bnd trngl} is bounded by
\begin{equation*}
\ll \int\limits_{\frac{\pi R}{2}}^{\pi m } r_{N}^{M,0}(x)^2 dx \ll
\frac{1}{R} + \frac{1}{\sqrt{M}},
\end{equation*}
as before. We bound the first summand of \eqref{eq:J32 bnd trngl} as
\begin{equation*}
\begin{split}
&\ll \int\limits_{\frac{\pi R}{2}}^{\pi m } \big| \E |V_{2}| - \E
{|Y_{N}^{M}}'(x)| \big| dx \ll \int\limits_{\frac{\pi R}{2}}^{\pi m
} \bigg[\sqrt{{\lambda_{2}^{M}}'} - \bigg({\lambda_{2}^{M}}'-
\frac{{r_{N}^{M,0}}'(x)^2}{1-r_{N}^{M,0}(x)^2}\bigg)^{1/2}\bigg]dx
\\&\ll \frac{1}{\sqrt{{\lambda_{2}^{M}}'}} \int\limits_{\frac{\pi R}{2}}^{\pi m } \frac{{r_{N}^{M,0}}'(x)^2}{1-r_{N}^{M,0}(x)^2}
\ll \int\limits_{\frac{\pi R}{2}}^{\pi m } {r_{N}^{M,0}}'(x)^2 dx
\ll \frac{1}{\sqrt{M}} + \frac{1}{R},
\end{split}
\end{equation*}
as earlier, since ${\lambda_{2,N}^{M}}'$ is bounded away from $0$,
and $r_{N}^{M,0}$ is bounded away from $1$ on the domain of the
integration. The second summand of \eqref{eq:J32 bnd trngl} is
bounded similarly, resulting in the same bound. Therefore
\begin{equation}
\label{eq:J32 bnd} J_{3,2} \ll \frac{1}{\sqrt{M}} + \frac{1}{R}.
\end{equation}

Recalling \eqref{eq:E32 bnd sumint}, the estimates \eqref{eq:J31
bnd} and \eqref{eq:J32 bnd} imply
\begin{equation*}
\frac{E_{3,2}}{N} \ll \frac{1}{\sqrt{M}} + \frac{1}{R}
\end{equation*}
by \eqref{eq:E32 bnd sumint}. Bounding $E_{3,1}$ and $E_{3,3}$ in a
similar (but easier) way, we get (see \eqref{eq:E3=E31-2E32+E33})
\begin{equation}
\label{eq:E3/N bnd} \frac{E_{3}}{N} \ll \frac{1}{\sqrt{M}} +
\frac{1}{R}
\end{equation}

\subsubsection{Collecting all the estimates}

Collecting the estimates \eqref{eq:E1/N bnd}, \eqref{eq:E2/N bnd}
and \eqref{eq:E3/N bnd}, we see that
\begin{equation*}
\begin{split}
&\frac{\Var(Z_{N}^{+}-Z_{N}^{M,+})}{N} =
\frac{E_{1}}{N}+\frac{E_{2}}{N}+\frac{E_{3}}{N} \\&= O\bigg(\epsilon
S+\frac{S}{\sqrt{M}\epsilon^2}+\frac{1}{S}\bigg)+ O\bigg( \frac{R
S^{2A} T^6}{M^{1/6}} + \frac{RS}{T}\bigg) +O\bigg(\frac{1}{\sqrt{M}}
+ \frac{1}{R}\bigg),
\end{split}
\end{equation*}
which could be made arbitrarily small, upon making an appropriate
choice for the parameters $\epsilon$, $S$, $R$ and $T$. Proposition
\ref{prop:E(Z-Zmol)^2=o(N)} is now proved.

\end{proof}

\appendix

\section{The third moment of $Z^{M}$ on short intervals is bounded}

\begin{proposition}
\label{prop:third mom ZM bnd} Let $L$ be a constant, $K=NL$ and for
$1 \le k \le K$ let $Z_{N,k}^{M}$ be the number of zeros of
$Y_{N}^{M}$ on $[(k-1)\frac{\pi m}{K},k\frac{\pi m}{K}]$. Then for
$L$ sufficiently large, all the third moments $\E (Z_{N,k}^{M})^3$
are uniformly bounded by a constant, independent of $N$ and $k$.
\end{proposition}

\begin{proof}
By stationarity, we may assume that $k=1$. For any $\tau>0$ let
$Z=Z_{N}^{M}(\tau)$ be the number of zeros of $Y=Y_{N}^{M}$ on
$[0,\tau]$. Since $L$ is arbitrarily large, we may reduce the
statement of the present Proposition to bounding $\E
Z_{N}^{M}(\tau)^3$, for $\tau>0$ sufficiently small. It will be
convenient to use the shortcut $r=r_{N}^{M}$.

Using the formula for the high combinatorial moments of the number
of crossings of stationary processes ~\cite{CL} (see remarks
\ref{rem:CL mom form just} and \ref{rem:VarZ=intE|X'(0)X'(t)||0}),
we obtain the bound
\begin{equation*}
\E\big[Z(Z-1)(Z-2)\big] \ll \iint\limits_{[0,\tau]^2} P(x,y)dxdy
\end{equation*}
for the third combinatorial moment (the number of triples) of
$Z(\tau)$, where $P=P_{N}^{M}$ is given by
\begin{equation}
\label{eq:P(x,y) def} P(x,y) = \frac{\E\big[|Y'(0)Y'(x)Y'(y)| \big|
Y(0)=Y(x)=Y(y)=0 \big]}{(2\pi)^{3/2}\sqrt{f(x,y)}}
\end{equation}
with $f(x,y)=f_{N}^{M}(x,y)=\det{A}$ with
\begin{equation*}
A=\left(\begin{matrix} 1 &r(x) &r(y) \\ r(x) &1 &r(y-x)
\\ r(y) &r(y-x) &1\end{matrix}\right).
\end{equation*}

It is easy to compute $f$ explicitly as
\begin{equation}
\label{eq:detA(x,y) expl} f(x,y) =
1-r(x)^2-r(y)^2-r(y-x)^2+2r(x)r(y)r(y-x).
\end{equation}

Since $$\E Z=\frac{\tau}{\pi} \sqrt{-r''(0)} = O(1),$$ and we proved
that the second moments $\E Z^2$ are uniformly bounded, while
proving Proposition \ref{prop:CLT for Zmol} (see condition
\ref{it:Brk cond2} of Theorem \ref{thm:Berk}), it is then sufficient
to prove that the function $P(x,y)$ is uniformly bounded near the
origin. Denote the random vector
$$(V_{1},V_{2},V_{3}) = (Y'(0),Y'(x),Y'(y))$$ conditioned upon
$Y(0)=Y(x)=Y(y)=0$. The random vector $(V_1,V_2,V_3)$ has a mean
zero multivariate Gaussian distribution and we have
\begin{equation}
\label{eq:P(x,y) def exp} P(x,y)=
\frac{\E[|V_{1}V_{2}V_{3}|]}{\sqrt{f(x,y)}}.
\end{equation}
by the definition \eqref{eq:P(x,y) def}.

Applying the Cauchy-Schwarts inequality twice implies the bound
\begin{equation*}
\E[|V_{1}V_{2}V_{3}|] \le (\E V_{1}^2)^{1/2} (\E V_{2}^4)^{1/4} (\E
V_{3}^4)^{1/4}.
\end{equation*}
Let $\mathcal{V}_{i}=\mathcal{V}_{i}(x,y)$ be the variance of
$V_{i}$ for $i=1,2,3$. Computing explicitly, we have
$\mathcal{V}_{i} = \frac{1}{f(x,y)} \mathcal{R}_{i}(x,y)$, where
\begin{equation*}
\begin{split}
\mathcal{R}_{1}(x,y) &:= \lambda_{2}'\det{A} -
r'(x)^2(1-r(y)^2)-r'(y)^2(1-r(x)^2)\\&-2
r'(x)r'(y)(r(x)r(y)-r(y-x)),
\end{split}
\end{equation*}
\begin{equation*}
\begin{split}
\mathcal{R}_{2}(x,y) &:= \lambda_{2}'\det{A}
-r'(x)^2(1-r(y-x)^2)-r'(x-y)^2(1-r(x)^2)\\&-2
r'(x)r'(x-y)(r(x)r(y-x)-r(y)),
\end{split}
\end{equation*}
and
\begin{equation*}
\begin{split}
\mathcal{R}_{3}(x,y) &:= \lambda_{2}'\det{A} -
r'(y)^2(1-r(y-x)^2)-r'(y-x)^2(1-r(y)^2)\\&-2r'(y)r'(y-x)(r(y)r(y-x)-r(x)),
\end{split}
\end{equation*}
where, as usual, we denote $\lambda_{2}' := -r''(0)$.

We then have
\begin{equation}
\label{eq:EV1V2V3/sqrt(f)<<sqrt R1R2R3/f^2}
\frac{\E[|V_{1}V_{2}V_{3}|]}{\sqrt{f(x,y)}} \ll
\frac{\sqrt{\mathcal{R}_{1}\mathcal{R}_{2}\mathcal{R}_{3}
}}{f(x,y)^{2}}.
\end{equation}

The uniform boundedness of $P(x,y)$ around the origin $(x,y)=(0,0)$
then follows from applying Lemmas \ref{lem:detA >> x2y2(y-x)2} and
\ref{lem:3rd mom kern num} on \eqref{eq:EV1V2V3/sqrt(f)<<sqrt
R1R2R3/f^2}, bearing in mind \eqref{eq:P(x,y) def exp}.

\end{proof}

\begin{lemma}
\label{lem:detA >> x2y2(y-x)2} Let $f(x,y)$ be defined by
\eqref{eq:detA(x,y) expl} with $r=r_{N}^{M}$. Then
\begin{equation*}
f(x,y) \gg x^2y^2(y-x)^2
\end{equation*}
uniformly w.r.t. $N$ in some (fixed) neighbourhood of the origin.
\end{lemma}

\begin{proof}

Recall that $r_{N}^{M}(x)=r_{N}(x)\cdot S_{M}(x)$, and we assume
that the neighbourhood is sufficiently small so that $S_{M}$ is
given by a single polynomial \eqref{eq:S_M=1+...} of degree $7$ in
$\frac{|x|}{M}$. We assume with no loss of generality, that $x,y>0$,
and furthermore, that $y>x$. Let
\begin{equation*}
\theta_{m}^{M}(x):=r_{N}^{M}(x)-1,
\end{equation*}
and
\begin{equation*}
\theta_{m}(x):=r_{N}(x)-1.
\end{equation*}
Let also
\begin{equation*}
\theta_{\infty}(x) = \frac{\sin{x}}{x}-1
\end{equation*}
be the limiting function. We will omit the parameters $m$ and $M$,
whenever there is no ambiguity.

We rewrite the definition of $f(x,y)$ as
\begin{equation*}
\begin{split}
f(x,y) &= -\big(\theta(x)^2+\theta(y)^2+\theta(y-x)^2\big) \\&+
2\big(\theta(x)\theta(y)+\theta(y)\theta(y-x)
+\theta(x)\theta(y-x)\big)+2\theta(x)\theta(y)\theta(y-x).
\end{split}
\end{equation*}

It is easy to Taylor expand $\theta=\theta_{m}^{M}(x)$ as
\begin{equation*}
\theta(x) = a_{2,m}^{M}x^2+O(x^4),
\end{equation*}
where the constant in the `$O$'-notation is universal, and
\begin{equation*}
 a_{2,m}^{M} =
a_{2}\bigg(1+O\bigg(\frac{1}{m}+\frac{1}{M^2}\bigg)\bigg),
\end{equation*}
where $a_{2}=-\frac{1}{6}$ is the corresponding Taylor coefficient
of the limiting function $\theta_{\infty}$. We rewrite it as
\begin{equation*}
\theta(x) =
-\frac{1}{6}x^2\bigg(1+O\bigg(x^2+\frac{1}{m}+\frac{1}{M^2}\bigg)\bigg),
\end{equation*}
so that
\begin{equation}
\label{eq:tripple prod}  \theta(x)\theta(y)\theta(y-x) =
-\frac{1}{6^3}
x^2y^2(y-x)^2\bigg(1+O\bigg(x^2+y^2+(y-x)^2+\frac{1}{m}+\frac{1}{M^2}\bigg)\bigg).
\end{equation}

Thus, it remains to estimate
\begin{equation*}
f_{2}(x,y) := 2\big(\theta(x)\theta(y)+\theta(y)\theta(y-x)
+\theta(x)\theta(y-x)\big) -(\theta(x)^2+\theta(y)^2+\theta(y-x)^2).
\end{equation*}
Let
\begin{equation}
\label{eq:theta inf Tay exp} \theta_{\infty}(x) =
\sum\limits_{n=2}^{\infty} a_{n} x^{n},
\end{equation}
where it is easy to compute $a_{n}$ to be
\begin{equation}
\label{eq:an comp} a_{n} = \begin{cases} \frac{(-1)^{n}}{(n+1)!} ,\;
&n \text{ even}\\0 &\text{otherwise}\end{cases}.
\end{equation}
Similarly, we expand $\theta_{m}$ and $\theta_{m}^{M}$ into Taylor
series
\begin{equation*}
\theta_{m}(x) = \sum\limits_{n=2}^{\infty} a_{n,m} x^{n},
\end{equation*}
and
\begin{equation*}
\theta_{m}^{M}(x) = \sum\limits_{n=2}^{\infty} a_{n,m}^{M} x^{n}.
\end{equation*}

We need the following estimates concerning the Taylor coefficients
of $\theta_{m}$ and $\theta_{m}^{M}$.

\begin{lemma}
\label{lem:anm, anmM}
\begin{enumerate}
\item \label{it:anm=an(1+O())}
We have the following estimates for the coefficients of
$\theta_{m}$,
\begin{equation*}
a_{2n,m} =
a_{2n}\bigg(1+O\bigg(\frac{1}{m}+\frac{n^2}{m^2}\bigg)\bigg)
\end{equation*}
for $n \ll m$ and
\begin{equation*}
a_{2n,m},a_{2n} \ll e^{4\pi m} \bigg(\frac{1}{4\pi m}\bigg)^{2n}
n^{O(1)}
\end{equation*}
for $n\gg m$. We have $a_{2n+1,m}=0$ for every $n$.

\item \label{it:anmM=an(1+O())}

We have the following estimates for the coefficients of
$\theta_{m}^{M}$,
\begin{equation*}
\begin{split}
a_{2n,m}^{M} &=
a_{2n}\bigg(1+O\bigg(\frac{1}{m}+\frac{n^2}{m^2}\bigg)\bigg) +
O\bigg(\frac{1}{M^2(2n-5)!}\bigg)
\end{split}
\end{equation*}
\begin{equation}
\label{eq:anmM est n<<m}
a_{2n+1,m}^{M} = \begin{cases} 0, &n\le 2 \\
O\big(\frac{1}{M^7 (2n-6)!} \big), & n\ge 3 \end{cases},
\end{equation}
for $n \ll m$, and
\begin{equation}
\label{eq:anmM est n>>m} a_{n,m}^{M} \ll e^{4\pi m}
\bigg(\frac{1}{4\pi m}\bigg)^{2n} n^{O(1)}
\end{equation}
for $n \gg m$.

\end{enumerate}

\end{lemma}

We postpone the proof of Lemma \ref{lem:anm, anmM} until after the
end of the proof of Lemma \eqref{lem:detA >> x2y2(y-x)2}.

We write
\begin{equation*}
\begin{split}
f_{2,m}^{M}(x,y) &= \sum\limits_{\substack{i,j=2 \\ i,j\ne
3,5}}^{\infty} a_{i,m}^{M} a_{j,m}^{M} \cdot
\big(2(x^{i}y^{j}+y^{i}(y-x)^{j}+x^{i}(y-x)^{j})\\&-(x^{(i+j)}+y^{(i+j)}+(y-x)^{(i+j)})
\big)
\\&= \sum\limits_{\substack{i,j=2 \\ i,j\ne
3,5}}^{\infty} a_{i,m}^{M} a_{j,m}^{M} \cdot
\big(x^{i}y^{j}+x^{j}y^{i}+y^{i}(y-x)^{j}+y^{j}(y-x)^{i}\\&+x^{i}(y-x)^{j}+x^{j}(y-x)^{i}-x^{(i+j)}-y^{(i+j)}-(y-x)^{(i+j)}
\big),
\end{split}
\end{equation*}
adding the summands corresponding to $(i,j)$ and $(j,i)$. We
introduce the polynomials
\begin{equation}
\label{eq:Fij def}
\begin{split}
F_{ij}(x,y) &:=
x^{i}y^{j}+x^{j}y^{i}+y^{i}(y-x)^{j}+y^{j}(y-x)^{i}+x^{i}(y-x)^{j}\\&+x^{j}(y-x)^{i}-x^{(i+j)}-y^{(i+j)}-(y-x)^{(i+j)}\in\Z[x,y],
\end{split}
\end{equation}
so that
\begin{equation*}
f_{2,m}^{M}(x,y) = \sum\limits_{\substack{i,j=2 \\ i,j\ne
3,5}}^{\infty} a_{i,m}^{M}a_{i,m}^{M} F_{i,j}(x,y)
\end{equation*}
and
\begin{equation*}
f_{2,\infty}(x,y) = \sum\limits_{\substack{i,j=2 \\ i,j \text{
even}}}^{\infty} a_{i}a_{j} F_{i,j}(x,y).
\end{equation*}

Note that $F_{2,2} = 0$ and for every even tuple $(i,j)\ne (2,2)$,
\begin{equation*}
x^2 y^2 (y-x)^2  | F_{i,j}(x,y),
\end{equation*}
so that in this case, we may define
\begin{equation*}
H_{i,j}(x,y) = \frac{F_{2i,2j}(x,y)}{x^2 y^2 (y-x)^2} \in \Z[x,y].
\end{equation*}
It is easy to compute $H_{2,4}$ and $H_{4,2}$ to be
\begin{equation*}
H_{2,4}(x,y)=H_{4,2}(x,y) =-6.
\end{equation*}

We claim that for $|x|,|y| \le \tau$, if $\tau$ is sufficiently
small,
\begin{equation}
\label{eq:Hij bnd} |H_{i,j}(x,y)| \ll \max\{2|x|,2|y|\}^{2(i+j-3)}
\end{equation}
To prove our claim, we write
\begin{equation*}
\begin{split}
\frac{F_{2i,2j}(x,y)}{(y-x)^2} &= \big(
y^{2i}(y-x)^{2(j-1)}+y^{2j}(y-x)^{2(i-1)}+x^{2i}(y-x)^{2(j-1)}\\&+x^{2j}(y-x)^{2(i-1)}-(y-x)^{2(i+j-1)}\big)-
\frac{(y^{2j}-x^{2j})}{y-x}\frac{(y^{2i}-x^{2i})}{y-x}.
\end{split}
\end{equation*}
The sum of the coefficients of the homogeneous monomials of the last
polynomial is $\ll 2^{2(i+j-1)}$. Since dividing by $x^2y^2$ does
not increase the coefficients, the $H_{i,j}$ are bounded by
$$\ll 2^{2(i+j-1)}\cdot\max\{|x|,|y|\}^{2(i+j-3)} \ll
\max\{2|x|,2|y|\}^{2(i+j-3)},$$ which is our claim \eqref{eq:Hij
bnd}.

We then have
\begin{equation*}
\frac{f_{2,\infty}(x,y)}{x^2 y^2 (y-x)^2} = -12
a_{2}a_{4}+\sum\limits_{i+j \ge 4} a_{2i}a_{2j} H_{i,j}(x,y),
\end{equation*}
so that on any fixed neighbourhood of the origin,
\begin{equation*}
\begin{split} &\bigg|\frac{f_{2,\infty}(x,y)}{x^2 y^2 (y-x)^2}+12 a_{2}a_{4}
\bigg| \le \sum\limits_{i+j\ge 4} |a_{2i} a_{2j}| |H_{i,j}(x,y)|
\\&\ll (x^2+y^2)\sum\limits_{i+j\ge 4}
\frac{1}{(2i+1)!}\frac{1}{(2j+1)!} 2^{2(i+j-3)}  \ll (x^2+y^2),
\end{split}
\end{equation*}
by \eqref{eq:an comp} and \eqref{eq:Hij bnd}.

Thus to finish the proof of Lemma \ref{lem:detA >> x2y2(y-x)2} it is
sufficient to bound
\begin{equation*}
\bigg|\frac{f_{2,\infty}(x,y)-f_{2,m}^{M}(x,y)}{x^2 y^2 (y-x)^2}
\bigg|.
\end{equation*}

We have
\begin{equation*}
\bigg|\frac{f_{2,\infty}(x,y)-f_{2,m}^{M}(x,y)}{x^2 y^2 (y-x)^2}
\bigg| \le \sum\limits_{i+j\ge 6}|a_{i,m}^{M}a_{j,m}^{M} -
a_{i}a_{j}| \bigg|\frac{F_{i,j}(x,y)}{x^2 y^2 (y-x)^2}\bigg| =
\Sigma^{even}+\Sigma^{odd}+2\Sigma^{mixed},
\end{equation*}
where
\begin{equation*}
\Sigma^{even} := \sum\limits_{i,j \text{ even}};\quad \Sigma^{odd}
:= \sum\limits_{i,j \text{ odd}} ;\quad \Sigma^{mixed} :=
\sum\limits_{i \text{ odd, } j \text{ even}}.
\end{equation*}

We have
\begin{equation*}
\begin{split}
|\Sigma^{even}| &= \sum\limits_{i+j \ge 3 }^{\infty}
|a_{2i,m}^{M}a_{2j,m}^{M} - a_{2i}a_{2j}| |H_{i,j}(x,y)| \\&\ll
\sum\limits_{i,j\ll m} \bigg( (\frac{1}{m} +
\frac{i^2+j^2}{m^2})\frac{1}{(2i+1)!(2j+1)!} + \frac{1}{M^2 \cdot
(2j-5)! (2i-5)!}\bigg) \cdot \tau^{2(i+j-3)} \\&+
\sum\limits_{\substack{ i \ll m\\ j\gg m}} \frac{1}{(2i-5)!} \cdot
\frac{e^{4\pi m}}{(4\pi
m)^{2j}} j^{O(1)} \tau^{2(i+j-3)} \\&+ \sum\limits_{\substack{ i \gg m\\
j\gg m}}  \frac{e^{8\pi m}}{(4\pi m)^{2(i+j)}} (i\cdot j)^{O(1)}
\tau^{2(i+j-3)} \ll \bigg(\frac{1}{m}+\frac{1}{M^2}\bigg),
\end{split}
\end{equation*}
which is sufficient.

The main problem while treating the odd and the mixed terms is that
for such a tuple $(i,j)$, $F_{i,j}(x,y)$ is not necessarily
divisible by $x^2y^2(y-x)^2$. However, in any case, it is divisible
by $x^2(y-x)^2$. We recall then our assumption $y>x$, and thus
\begin{equation*}
\bigg| \frac{F_{i,j}}{x^2y^2 (y-x)^2} \bigg| = \bigg|
\frac{1}{y^2}\cdot\frac{F_{i,j}}{x^2 (y-x)^2} \bigg|.
\end{equation*}
We then define for an odd $i \ge 7$ and any $j$ (i.e. $j\ge 2$ and
$j\ne 3,5$) the polynomial
\begin{equation*}
G_{i,j}(x,y) = \frac{F_{i,j}}{x^2 (y-x)^2}.
\end{equation*}
The polynomial $G_{i,j}$ is a homogeneous polynomial of degree
$i+j-4$ and we have
\begin{equation}
\label{eq:Gij bnd} G_{i,j}(x,y) \ll \max\{2|x|,2|y| \} ^{i+j-4},
\end{equation}
similarly to \eqref{eq:Hij bnd}. Thus, one has
\begin{equation}
\label{eq:Gij/y^2 bnd} \bigg|\frac{1}{y^2} G_{i,j}(x,y) \bigg| \ll
\max\{2|x|,2|y| \} ^{i+j-6}
\end{equation}
(here we use $|x|\le |y|$).

We then write
\begin{equation*}
\Sigma^{odd} = \sum a_{i,m}^{M} a_{j,m}^{M}
\frac{G_{i,j}(x,y)}{y^{2}} \ll \frac{1}{M^{14}}+\frac{1}{m^2},
\end{equation*}
using the same approach as in case of $\Sigma^{even}$, this time
plugging \eqref{eq:Gij/y^2 bnd}. Similarly, one obtains the estimate
\begin{equation*}
\Sigma^{mixed} =  \frac{1}{M^7}+\frac{1}{m^2}.
\end{equation*}

Combining \eqref{eq:tripple prod} with the estimates on
$f_{2,m}^{M}$ shows that
\begin{equation*}
\frac{f(x,y)}{x^2y^2(y-x)^2} =
a+O\bigg(x^2+y^2+\frac{1}{m}+\frac{1}{M^2}\bigg),
\end{equation*}
where $$a=-12a_{2}a_{4}-2/6^3= \frac{1}{135}.$$ This concludes the
proof of Lemma \ref{lem:detA >> x2y2(y-x)2} assuming Lemma
\ref{lem:anm, anmM}.

\end{proof}

\begin{proof}[Proof of Lemma \ref{lem:anm, anmM}]
First, it is clear that part \eqref{it:anmM=an(1+O())} of Lemma
\ref{lem:anm, anmM} follows from part \eqref{it:anm=an(1+O())} by
\eqref{eq:S_M=1+...}, which holds for $x>0$ sufficiently small.

We write
\begin{equation*}
\theta_{m}(x) = \frac{2m}{2m-1} \bigg(\frac{\sin{x}}{x}
h\bigg(\frac{x}{2m}\bigg) - 1\bigg),
\end{equation*}
where $h(x) = \frac{x}{\sin{x}}$. The multiplication by
$\frac{2m}{2m-1}$ poses no problem here. Now, the Taylor expansion
of $h(x)$ is well-known to be
\begin{equation*}
h(x) = \sum\limits_{n\ge 0} (-1)^{n+1} (2^{2n}-2) B_{2n}
\frac{x^{2n}}{(2n)!},
\end{equation*}
where $B_{n}$ are the Bernoulli numbers. Recalling the Taylor
expansion
$$\frac{\sin{x}}{x} = \sum\limits_{j\ge 0} (-1)^{j} \frac{x^{2j}}{(2j+1)!},$$
we obtain, after a little rearranging,
\begin{equation*}
\bigg(\frac{\sin{x}}{x} h\bigg(\frac{x}{2m}\bigg) - 1\bigg) =
-\sum\limits_{k\ge 1} \bigg\{\sum\limits_{\substack{j,n\ge 0 \\
j+n=k}} {2k+1 \choose 2n} (2^{2n}-2) \frac{B_{2n}}{(2m)^{2n}}
\bigg\} \frac{(-1)^k x^{2k}}{(2k+1)!}.
\end{equation*}
Now, $$B_{2n}\sim (-1)^{n-1} \cdot 2 \frac{(2n)!}{(2\pi)^{2n}},$$
and therefore the coefficient of $\frac{(-1)^k x^{2k}}{(2k+1)!}$ is
\begin{equation*}
1+O\bigg(\sum\limits_{n\ge 1} {2k+1 \choose 2n} \frac{(2n)!}{(4\pi
m)^{2n}} \bigg).
\end{equation*}
By comparing consecutive terms in this series, we find that this is
$1+O(k^2/m^2)$ provided that $k \ll m$, and it is
\begin{equation*}
\ll e^{4\pi m} \bigg(\frac{2k+1}{4e\pi m}\bigg) ^{2k} k^{O(1)}
\end{equation*}
if $k\gg m$, using Stirling's formula, since the maximal term occurs
when $$2n\approx 2k+1-4\pi m.$$ Hence we have
\begin{equation*}
a_{2n,m} = a_{2n}
\bigg(1+O\bigg(\frac{1}{m}+\frac{n^2}{m^2}\bigg)\bigg)
\end{equation*}
for $n\ll m$, and
\begin{equation*}
a_{2n,m},a_{2n} \ll e^{4\pi m} \bigg(\frac{1}{4\pi m} \bigg) ^{2n}
n^{O(1)}
\end{equation*}
for $n \gg m$.

\end{proof}

\begin{lemma}
\label{lem:3rd mom kern num} Let
\begin{equation*}
\begin{split}
\mathcal{R}_{1}(x,y) &:= \lambda_{2}' f(x,y) -
r'(x)^2(1-r(y)^2)-r'(y)^2(1-r(x)^2)\\&-2
r'(x)r'(y)(r(x)r(y)-r(y-x)),
\end{split}
\end{equation*}

\begin{equation*}
\begin{split}
\mathcal{R}_{2}(x,y) &:= \lambda_{2}' f(x,y)
-r'(x)^2(1-r(y-x)^2)-r'(x-y)^2(1-r(x)^2)\\&-2
r'(x)r'(x-y)(r(x)r(y-x)-r(y)),
\end{split}
\end{equation*}
and
\begin{equation*}
\begin{split}
\mathcal{R}_{3}(x,y) &:= \lambda_{2}' f(x,y) -
r'(y)^2(1-r(y-x)^2)-r'(y-x)^2(1-r(y)^2)\\&-2r'(y)r'(y-x)(r(y)r(y-x)-r(x)),
\end{split}
\end{equation*}
where $r=r_{N}^{M}$ and $f$ is defined by \eqref{eq:detA(x,y) expl}.
Then
\begin{equation*}
\mathcal{R}_{1}(x,y) = O(x^4y^4 (y-x)^2),
\end{equation*}
\begin{equation*}
\mathcal{R}_{2}(x,y) = O(x^4y^2 (y-x)^4),
\end{equation*}
and
\begin{equation*}
\mathcal{R}_{3}(x,y) = O(x^2y^4 (y-x)^4).
\end{equation*}
\end{lemma}

\begin{proof}
We prove the first statement only, the other ones being symmetrical.
We will assume that $x,y > 0$ and moreover $y>x$, so that $S_{M}$ is
given by a polynomial \eqref{eq:S_M=1+...} of degree $7$ in
$\frac{x}{M}$. For brevity, we denote $h(x,y) = h_{m}^{M}(x,y) =
\mathcal{R}_{1}(x,y)$. Similarly, we denote $h_{\infty}(x,y)$,
defined the same way as $h_{m}^{M}(x,y)$, where we use
$r(x)=r_{\infty}(x)=\frac{\sin{x}}{x}$ rather than $r_{N}^{M}$. We
rewrite $h(x,y)$ in terms of $\theta:=r-1$ as
\begin{equation*}
\begin{split}
h(x,y) &:= \lambda_{2}' f(x,y) + 2\bigg(\theta'(x)^2
\theta(y)+\theta'(y)^2\theta(x)\\&-\theta'(x)\theta'(y)\big(\theta(x)+
\theta(y)-\theta(y-x)\big)\bigg)
\\&+(\theta(x)\theta'(y)-\theta(y)\theta'(x))^2.
\end{split}
\end{equation*}

Expanding $\theta$ and $\theta'$ into the Taylor series as earlier
and using $\lambda_{2}'=-2a_{2}$, we have
\begin{equation}
\label{eq:V1inf sum pol}
\begin{split}
h_{\infty}(x,y) &= \sum\limits_{i_{1},j_{1},i_{2},j_{2}\in S}
a_{2i_{1}}a_{2j_{1}}a_{2i_{2}}a_{2j_{2}}
A_{2i_{1},2j_{1},2i_{2},2j_{2}}(x,y) \\&+ \sum\limits_{i,j,k\ge 2}
a_{2i}a_{2j}a_{2k} B_{2i,2j,2k}(x,y) + \sum\limits_{i,j \ge 2, k\ge
1} a_{2}a_{2i}a_{2j}a_{2k}C_{2i,2j,2k}(x,y) \\&+ \sum\limits_{i,j
\ge 2} a_{2}^2 a_{2i} a_{2j}D_{2i,2j}(x,y)+ \sum\limits_{i,j\ge 2}
a_{2}a_{2i}a_{2j} E_{2i,2j}(x,y) + \sum\limits_{i\ge 2}
a_{2}^2a_{2i}I_{2i}(x,y)
\\&+\sum\limits_{i\ge 2} a_{2}^3a_{2i} J_{2i}(x,y),
\end{split}
\end{equation}
where
\begin{equation*}
A_{i_{1},j_{1},i_{2},j_{2}}(x,y) = j_{1}j_{2} \cdot
(x^{i_{1}}y^{j_{1}-1}-x^{j_{1}-1}y^{i_{1}}) \cdot
(x^{i_{2}}y^{j_{2}-1}-x^{j_{2}-1}y^{i_{2}}),
\end{equation*}
\begin{equation*}
B_{i,j,k}(x,y) = 2ij \big( x^{i+j-2}y^{k}+y^{i+j-2}x^{k}
-x^{i-1}y^{j-1}(x^{k}+y^{k}-(y-x)^{k})\big),
\end{equation*}
\begin{equation*}
C_{i,j,k} = -4x^{i}y^{j}(y-x)^{k},
\end{equation*}
\begin{equation*}
\begin{split}
D_{i,j}(x,y) &= -4(y-x)^{j}(x^2y^{i}+x^{i}y^2) +4(yx^{i}-xy^{i})
(yx^{j}-xy^{j})\\&+4j(x^{i}y-xy^{i})(x^{2}y^{j-1}-x^{j-1}y^2),
\end{split}
\end{equation*}
\begin{equation*}
\begin{split}
E_{i,j}(x,y) &=
-2F_{i,j}+4i\big(2(x^{i}y^{j}+x^{j}y^{i})-(x^{i-1}y+y^{i-1}x)(x^{j}+y^{j}-(y-x)^{j})
\big) \\&+ 2ij\big(x^{i+j-2}y^2+y^{i+j-2}x^2-2x^{i}y^{j} \big),
\end{split}
\end{equation*}
(here the polynomial $F_{i,j}$ is defined as in \eqref{eq:Fij def}),
\begin{equation*}
I_{i}(x,y) =
-4F_{i,2}+8(x^2y^{i}+x^{i}y^{2}-xy(x^{i}+y^{i}-(y-x)^{i})) =0,
\end{equation*}
so that we may disregard $I_{i}$ altogether,
\begin{equation*}
\begin{split}
J_{i}(x,y) &=
-4x^{2}y^{2}(y-x)^{i}-4(y-x)^{2}(x^{2}y^{i}+x^{i}y^{2})\\&+8xy(yx^{i}-xy^{i})(x-y)+
4i(x-y)(x^{3}y^{i}-y^{3}x^{i}),
\end{split}
\end{equation*}
and
\begin{equation*}
S = \Z^{4} \setminus \bigg(\big(\{(i_{1},1)\} \times \{(i_{2},1) \}
\big)\cup \big(\{(i_{1},1)\}\times \{(1,j_{2}) \} \big) \cup
\big(\{(1,j_{1})\}\times \{ (i_{2},1) \} \big)\bigg).
\end{equation*}
From all the above, it is easy to check that for all the (even)
indexes within the frame, $A,C,D$ and $J$ are divisible by
$P(x,y):=x^4y^4(y-x)^2$, and, moreover,
$B_{i,j,k}(x,y)+B_{j,i,k}(x,y)$ and $E_{i,j}(x,y)+E_{j,i}(x,y)$ are
divisible by $P(x,y)$ (in particular, $B_{i,i,k}$ and $E_{i,i}$
are). It follows then that all the polynomials above vanish, unless
their degree is $\ge 10$. In addition, we have the following
estimates, which follow from the same reasoning as while proving
\eqref{eq:Gij bnd}:
\begin{equation*}
\begin{split}
&\bigg|\frac{A_{i_{1},j_{1},i_{2},j_{2}}(x,y)}{x^4y^4(y-x)^2}\bigg|,
\bigg|\frac{B_{i,j,k}(x,y)}{x^4y^4(y-x)^2}\bigg|,
\bigg|\frac{C_{i,j,k}}{x^4y^4(y-x)^2}\bigg|,
\bigg|\frac{D_{i,j}(x,y)}{x^4y^4(y-x)^2}\bigg|,
\\&\bigg|\frac{(E_{i,j}+E_{j,i})(x,y)}{x^4y^4(y-x)^2}\bigg|,
\bigg|\frac{J_{i}(x,y)}{x^4y^4(y-x)^2}\bigg| \ll
\max\{2^{1+\epsilon}|x|,2^{1+\epsilon}|y|\}^{d},
\end{split}
\end{equation*}
(say), for some $\epsilon >0$, where $d$ is the degree of the
corresponding polynomial on the LHS.

Thus, we have
\begin{equation}
\label{eq:h sim x4y4y-x2}
\bigg|\frac{h_{\infty}(x,y)}{x^4y^4(y-x)^2} -c\bigg| = O(x^2+y^2),
\end{equation}
by the rapid decay \eqref{eq:an comp} of the Taylor coefficients of
$\theta_{\infty}$. The constant $c$ may be computed explicitly to be
$$c=\frac{1}{212625}.$$

To conclude the proof of Lemma \ref{lem:3rd mom kern num} in this
case, we need to bound
$$\bigg|\frac{h_{\infty}(x,y) - h_{m}^{M}(x,y)}{x^4y^4(y-x)^2}
\bigg|.$$ Similarly to \eqref{eq:V1inf sum pol}, we have
\begin{equation*}
\begin{split}
h_{m}^{M}(x,y) &= \sum\limits_{i_{1},j_{1},i_{2},j_{2}\in S'}{}'
a_{i_{1},m}^{M}a_{j_{1},m}^{M}a_{i_{2},m}^{M}a_{j_{2},m}^{M}
A_{i_{1},j_{1},i_{2},j_{2}}(x,y) \\&+ \sum\limits_{i,j,k\ge 4}{}'
a_{i,m}^{M}a_{j,m}^{M}a_{k,m}^{M} B_{i,j,k}(x,y) + \sum\limits_{i,j
\ge 4, k\ge 2}{}'
a_{2,m}^{M}a_{i,m}^{M}a_{j,m}^{M}a_{k,m}^{M}C_{i,j,k}(x,y) \\&+
\sum\limits_{i,j \ge 4}{}' {(a_{2,m}^{M})}^2 a_{i,m}^{M}
a_{j,m}^{M}D_{i,j}(x,y)+ \sum\limits_{i,j\ge 4}{}' a_{2,m}^{M}
a_{i,m}^{M}a_{j,m}^{M} E_{i,j}(x,y) \\&+ \sum\limits_{i\ge 4}{}'
{(a_{2,m}^{M})}^3a_{i,m}^{M} J_{i}(x,y),
\end{split}
\end{equation*}
where the $'$ in the $\sum'$ stands for the fact that all of the
indexes involved in the summations above are $\ne 3,5$, and
\begin{equation*}
S = \Z^{4} \setminus \bigg(\{(i_{1},2)\} \times \{(i_{2},2) \}  \cup
\{(i_{1},2)\}\times \{(2,j_{2}) \} \cup \{(2,j_{1})\times \{
(i_{2},2) \}\bigg).
\end{equation*}

As in the proof of Lemma \ref{lem:detA >> x2y2(y-x)2}, the tricky
part here is that for odd indexes, the polynomials are no longer
divisible by $P(x,y)=x^4y^4(y-x)^2$. However, we notice that, by our
assumption $y\ge x$, it is sufficient that they are divisible by
$x^4(y-x)^2$, and their degree is $\ge 10$ (see the end of the proof
of Lemma \ref{lem:detA >> x2y2(y-x)2}). To check this, we note that
there is certainly no problem with
$A_{i_{1},j_{1},i_{2},j_{2}}(x,y)$, $(B_{i,j,k}+B_{j,i,k})(x,y)$ and
$C_{i,j,k}(x,y)$. Moreover, one can easily check that $D_{i,j}$ and
$E_{i,j}+E_{j,i}$ are divisible by $x^4(y-x)^2$ for any $i,j\ge 2$,
and for any $i\ge 5$, $J_{i}$ is divisible by $x^4y^2(y-x)^2$.
Finally, we check that all the relevant polynomials have degree $\ge
10$.

This, together with the rapid decay \eqref{eq:anmM est n<<m} for
$n\ll m$ and $\eqref{eq:anmM est n>>m}$ for $n\gg m$, of the Taylor
coefficients, imply that
\begin{equation*}
\bigg|\frac{h_{\infty}(x,y)-h_{m}^{M}(x,y)}{x^4y^4(y-x)^2} \bigg|
\ll \frac{1}{m}+\frac{1}{M^2}.
\end{equation*}
Combining the last estimate with \eqref{eq:h sim x4y4y-x2} concludes
the proof of the present Lemma.

\end{proof}

\end{document}